\documentclass[journal,final,leqno,onefignum,onetabnum]{siamltex1213}
\usepackage{amsmath}
\usepackage{mathrsfs,amssymb,amsmath,booktabs,array,xcolor,url,cite,color,soul,multirow,multicol}
\usepackage{mathbbold,bbm}
\usepackage{slashbox}
\usepackage{stfloats}
\usepackage{amsfonts}
\usepackage{extarrows}
\usepackage{textcomp}
\usepackage{ulem}
\usepackage{cases}
\usepackage{bm}
\usepackage{chemarrow,chemfig}

\newtheorem{thm}{\textit{Theorem}}
\newtheorem{lem}{\textit{Lemma}}
\newtheorem{rem}{\textit{Remark}}
\newtheorem{cor}{\textit{Corollary}}
\newtheorem{result}{\textit{Result}}
\newtheorem{pro}{\textit{Proposition}}

\newcommand{\Rmnum}[1]{\expandafter\@slowromancap\romannumeral #1@}

\sethlcolor{yellow}

\def\dd{\text{d}}

\title{Stochastic Weak Passivity Based Stabilization of Stochastic Systems with Nonvanishing Noise\thanks{This work was supported by
the National Natural Science Foundation of China under Grant No.
11271326 and 61611130124, and the Research Fund for the Doctoral
Program of Higher Education of China under Grant No.
20130101110040.}}

\author{Zhou fang\footnotemark[2]
\and Chuanhou Gao \footnotemark[2]}

\begin{document}
\maketitle
\slugger{mms}{xxxx}{xx}{x}{x--x}

\renewcommand{\thefootnote}{\fnsymbol{footnote}}

\footnotetext[2]{School of Mathematical Sciences, Zhejiang University,
Hangzhou 310027, P. R. China. (\email{Zhou\_Fang@zju.edu.cn,
gaochou@zju.edu.cn (Correspondence)}).}

\renewcommand{\thefootnote}{\arabic{footnote}}

\begin{abstract}
{For stochastic systems with nonvanishing noise, i.e., at the desired state the noise port does not vanish, it is impossible to achieve the global stability of the desired state in the sense of probability. This bad property also leads to the loss of stochastic passivity at the desired state if a radially unbounded Lyapunov function is expected as the storage function. To characterize a certain (globally) stable behavior for such a class of systems, the stochastic asymptotic weak stability is proposed in this paper which suggests the transition measure of the state to be convergent and the ergodicity. By defining stochastic weak passivity that admits stochastic passivity only outside a ball centered around the desired state but not in the whole state space, we develop stochastic weak passivity theorems to ensure that the stochastic systems with nonvanishing noise can be globally$\backslash$ locally stabilized in weak sense through negative feedback law. Applications are shown to stochastic linear systems and a nonlinear process system, and some simulation are made on the latter further.}
\end{abstract}



\section{Introduction}
Stochastic phenomena have emerged universally in
many physical systems due to noise, disturbance and uncertainty. The
unpredictability to them leads to it a great challenge to stabilize
a stochastic system. During the past decades, the stabilization of
nonlinear stochastic systems had constituted one of central problems
in stochastic process control both theoretically and practically. A
great deal of methods emerge as the times require, among which
stochastic passivity based control is a popular one. Rooting in the
passivity theory \cite{Willems72,Desoer75} and the stochastic
version of Lyapunov theorem \cite{Khasminskii11}, the stochastic
passivity theory \cite{Florchinger99} was developed for
stabilization and control of nonlinear stochastic systems. By means
of state feedback laws, the asymptotic stabilization in probability
can be achieved for a stochastic affine system provided some rank
conditions are fulfilled and the unforced stochastic affine system
is Lyapunov stable in probability \cite{Florchinger99}. Following
this study, Lin et. al. \cite{Lin12} explored the relationship
between a stochastic passive system and the corresponding
zero-output system, and further established the global stabilization
results. Parallelizing to the development of stochastic passivity in
theory, Satoh et. al. \cite{Satoh13} applied this methodology to
port-Hamiltonian systems, and the solutions for stabilization of a
large class of nonlinear stochastic systems are thus available.
There are also some reports that stochastic passivity is applied to
$H_\infty$ filtering problem \cite{Zhang06} and controlling
stochastic mechanical systems \cite{Mehra12}.

{Despite the large success achieved, stochastic
passivity based control seems to only work under the condition that the
noise vanishes at the stationary solution (very often being at the
origin) if a radially unbounded Lyapunov function is expected as the storage function. This means that if a stochastic system has nonzero noise port at the stationary solution or has persistent noise port, such a
method may be out of action. One of the aims of this paper is to
derive the necessary conditions that a stochastic system is
stochastically passive, and further give the sufficient conditions
to say a stochastic system losing stochastic passivity.
Equivalently, we prove that there does not exist a radially unbounded Lyapunov
function rendering the stochastic system to be globally
asymptotically stable in probability provided the noise does not
vanish at the desired state. The ubiquitousness of such a class of systems in the mechanical\mbox{\cite{Satoh14,Satoh15}} and biological\mbox{\cite{Rufino-Ferreira12}} fields motivates us to define a kind of novel stability, termed as stochastic asymptotic weak stability, to characterize a certain (globally) stable behavior for them. The stochastic asymptotic weak stability requests the system state to be convergent in distribution and ergodic. The former means the state to evolve within a small region around the desired state in a large probability while the latter ensures that the state evolution almost always take place within this region.}

{On the face of it, the stochastic asymptotic weak stability is somewhat similar to the concept of stochastic bounded stability proposed in\mbox{\cite{Satoh14,Satoh15}} in that a stochastic system with persistent noise is considered for the same purpose. That concept also means that the state will evolve within a bounded
region around a desired state with a large probability which depends on the region radius. Especially when the region radius goes infinite, the probability will be one. However, there is evident difference between these two kinds of stability. Stochastic bounded stability cannot characterize the ergodicity of the state. Namely, once the trajectory of the state runs out of the bounded region with a small probability, the coming evolution will take place in a
larger bounded region to reach a ``new" stochastic bounded stability with a larger probability. In addition, the stochastic asymptotic weak stability is different from stochastic noise-to-state\mbox{\cite{Rufino-Ferreira12,Deng01}} and input-to-state stability\mbox{\cite{Liu08}} too. The latter two
kinds of stability also serves for characterizing the stable behavior of stochastic systems with nonvanishing noise. They describe the convergence of the expectation of the state, for which the transition measure is controlled by
defining a particular function. Comparatively speaking, they say nothing about the ergodicity of the state, and do not mean either that the state must evolve within a small region around the desired state. Therefore, the stochastic asymptotic weak stability is able to provide more details on characterizing the ``stable" evolution of the state.}

In the concept of stochastic asymptotic weak stability, the convergence in distribution describes the evolution trend of the probability distribution of the stochastic system under
consideration. As one may know, for a stochastic system the
probability density function satisfies the Fokker-Planck equation
\cite{Khasminskii11}. Hence, a usual way to achieve convergence in distribution often starts from analyzing the properties of the
solutions of the Fokker-Planck equation, including the existence,
uniqueness and convergence. Based on this equation, Zhu et al.
\cite{Zhu01,Zhu06} studied the exact stationary solution of
distribution density function for stochastic Hamiltonian systems.
Liberzon et. al. \cite{Liberzon2000} developed a feedback controller
to stabilize in distribution a class of nonlinear stochastic systems
for which the steady-state distribution density function can be
solved from the Fokker-Planck equation. In addition, probability
analysis is another way to serve for achieving the weak stability.
Zakai \cite{zakai69} presented a Lyapunov criterion to suggest the
existence of stationary probability distribution and the convergence
of transition probability measure for stochastic systems with
globally Lipschitzian coefficients. Stettner \cite{Stettner94}
pointed out that the strongly feller and irreducible process are
stable in distribution. Khasminskii \cite{Khasminskii60} constructed a
Markov chain to analyze the convergence of the probability
distribution, and further obtained the Markov process to be convergent in distribution \cite{Khasminskii11} if it is ``mix sufficiently well" in an open domain and the recurrent time is finite. The conditions that
renders the recurrent time to be finite give us large inspiration on
developing the stabilizing ways in weak sense for stochastic systems
with nonvanishing noise.

{In this paper, we will show that the recurrent property of a
stochastic system is highly relevant to the stochastic passivity
behavior. Based on this comparison, we define the stochastic
passivity not in the whole state space, but only outside a ball
centered around the desired state, which is labeled as stochastic
weak passivity in the context. Within the framework of stochastic
weak passivity, we do not need to care whether the noise port of a
stochastic system vanishes at the desired state or not. Therefore, it is suited to handle the stabilization issue of stochastic differential systems with nonvanishing noise. Further, we link the stochastic weak passivity with the stochastic asymptotic weak stability, and develop
stabilizing controllers using the stochastic weak passivity to
achieve the asymptotic weak stability of stochastic systems. The sufficient
conditions for global and local asymptotic stabilization in weak
sense are provided by means of negative feedback laws, respectively.}

The rest of the paper is organized as follows. Section $2$ presents
some preliminaries on stochastic passivity. In section $3$, the loss
of stochastic passivity is analyzed and the problem of interest is
formulated. In section $4$, we propose the framework of stochastic
weak passivity theory and make a link between stochastic weak
passivity and asymptotic weak stability. Some basic concepts and the main
results (expressed as two stochastic weak passivity theorem and one
refined version) for stabilizing stochastic systems in weak sense
are given in this section. Section $5$ illustrates the efficiency of
the stochastic weak passivity theory through two application
examples. Finally, section $6$ concludes this paper and makes a
prospect of future research.
\section{Preliminaries of stochastic passivity}
In this section, we will give a birds-eye view of mathematical
systems theory related to stochastic differential systems.

{We begin with a stochastic differential equation
written in the sense of It\^{o}
\begin{eqnarray}{\label{StochasticEquation}}
\dd\bm{x}=\bm{f}(\bm{x})\dd t+\bm{h}(\bm{x})\dd \boldsymbol{\omega}
\end{eqnarray}
where $\bm{x}\in \mathbb{R}^{n}$, $t\in \mathbb{R}_{\geq 0}$,
$\bm{f}:\mathbb{R}^{n}\mapsto \mathbb{R}^{n}$ and
$\bm{h}:\mathbb{R}^{n}\mapsto \mathbb{R}^{n\times r}$ are locally
Lipschitz continuous functions, and $\boldsymbol{\omega}\in
\mathbb{R}^r$ is a standard Wiener process defined on a complete
probability space. Assume $\bm{x}(t)$ to be the stochastic process
solution and $\bm{x}^*$ to be the equilibrium solution (if exists)
of Eq. (\ref{StochasticEquation}), then we have }

\begin{definition}[Transition Measure \cite{Khasminskii11}] The transition measure of
$\bm{x}(t)$, denoted by $\mathcal{P}(\cdot,\cdot,\cdot)$, is a
function from $\mathbb{R}_{\geq
0}\times\mathbb{R}^{n}\times\mathscr{B}$ to $[0,1]$ such that
\begin{eqnarray}{\label{TransitionMeasure}}
\mathcal{P}(t,\bm{x}_0,\mathbb{A})=\mathrm{P}\left(\bm{x}(t)\in
\mathbb{A} | \bm{x}(0)=\bm{x}_0\right)
\end{eqnarray}
where $\mathscr{B}$ is the $\sigma$-algebra of Borel sets in
$\mathbb{R}^{n}$, $\mathbb{A}\in \mathscr{B}$ is a Borel subset, and
$\mathrm{P}(\cdot)$ denotes the probability function.
\end{definition}

\begin{definition}[Invariant Measure \cite{Khasminskii11}]
Let $\pi$ be a measure defined on a Borel space $\mathscr{B}$, then
$\pi$ is an probability invariant measure for a stochastic system of
Eq. (\ref{StochasticEquation}) if $\pi(\mathbb{R}^{n})=1$ and
\begin{equation}\label{Def-InvariantMeasure}
\pi(\mathbb{A})=\int_{\mathbb{R}^{n}}\mathcal{P}(t,\bm{x},\mathbb{A})\pi(\dd\bm{x}),~~
\forall~ t\textgreater 0~ \rm{and}~\forall~\mathbb{A}\in \mathscr{B}
\end{equation}
\end{definition}

\begin{definition}[Stable in Probability \cite{Khasminskii11}] The equilibrium solution $\bm{x}^*$
of Eq. (\ref{StochasticEquation}) is \\~ $~~~~~~\mathrm{(1)}$ stable
in probability if
\begin{equation*}
\lim_{\bm{x}(0)\to
\bm{x}^*}\mathrm{P}\big(\sup{\parallel\bm{x}(t)-\bm{x}^*\parallel_2}<\epsilon\big)=1,~
\forall~ \epsilon\textgreater 0;
\end{equation*}
\\~ $~~~~~~\mathrm{(2)}$ locally asymptotically stable in probability if
\begin{equation*}
\lim_{\bm{x}(0)\to
\bm{x}^*}\mathrm{P}\left(\lim_{t\to\infty}{\parallel\bm{x}(t)-\bm{x}^*\parallel_2}=0\right)=1;
\end{equation*}
\\~ $~~~~~~\mathrm{(3)}$ globally asymptotically stable in probability if
\begin{equation*}
\mathrm{P}\left(\lim_{t\to\infty}{\parallel\bm{x}(t)-\bm{x}^*\parallel_2}=0\right)=1,~\forall~
\bm{x}(0).
\end{equation*}
\end{definition}

{In order to analyze the stability of stochastic
systems, the stochastic version of the second Lyapunov theorem and
passivity theorem were proposed in succession. }

\begin{thm}[Stochastic Lyapunov Theorem\cite{Khasminskii11}]\label{StochasticLyapunov}
If there exists a positive definite
$\mathscr{C}^2(\mathbb{D};\mathbb{R})$ function $V(\bm{x})$ with
respect to $\bm{x}-\bm{x}^*$ such that
\begin{equation}\label{LV}
\mathcal{L}\big[{V(\bm{x})}\big]\leq 0,~\forall~ \bm{x}\in\mathbb{D}
\end{equation}
then the equilibrium solution $\bm{x}^*$ of Eq.
(\ref{StochasticEquation}) is stable in probability, where
$\mathbb{D}\subseteq \mathbb{R}^{n}$ is a bounded open neighborhood
of $\bm{x}^*$ and $\mathcal{L}[\cdot]$ is the infinitesimal
generator of the solution of Eq. (\ref{StochasticEquation}),
calculated through
\begin{equation}{\label{InfinitesimalGenerator}}
\mathcal{L}[\cdot]=\frac{\partial (\cdot)}{\partial
\bm{x}}\bm{f}+\frac{1}{2}\mathrm{tr}\left\{\frac{\partial^{2}
(\cdot)}{\partial \bm{x}^{2}}\bm{h}\bm{h}^\top\right\}
\end{equation}

If the equality in Eq. (\ref{LV}) holds if and only if
$\bm{x}=\bm{x}^*$, then $\bm{x}^*$ is locally asymptotically stable
in probability.

Further, if $\mathbb{D}=\mathbb{R}^{n}$,
$\lim_{\parallel\bm{x}\parallel_2\to\infty}V(\bm{x})=\infty$ (often
said that the Lyapunov function $V(\bm{x})$ is {radially unbounded}) and
$\mathcal{L}\big[{V(\bm{x})}\big]=0\Leftrightarrow \bm{x}=\bm{x}^*$,
then $\bm{x}^*$ is globally asymptotically stable in probability.
\end{thm}

{The stochastic passivity theorem is not handed
directly from the literature, but it may be obtained immediately
from the definition of stochastic passivity.}

\begin{definition}[Stochastic Passivity \cite{Florchinger99}]\label{Defstochasticpassivity} An input-output
stochastic differential system in the sense of It\^{o}
\begin{eqnarray}{\label{StochasticSystem}}
\Sigma_S:~\left\{
\begin{array}{lll}
\dd\bm{x}&=&\bm{f}(\bm{x},\bm{u})\dd t+ \bm{h}(\bm{x},\bm{u})\dd \boldsymbol{\omega} \\
\bm{y}&=&\bm{s}(\bm{x},\bm{u})
\end{array}
\right.
\end{eqnarray}
is said to be stochastically passive if there exists a positive
semi-definite $\mathscr{C}^2(\mathbb{R}^{n};\mathbb{R})$ function
$S(\bm{x})$ such that
\begin{equation}
\mathcal{L}\big[{S(\bm{x})}\big]\leq \bm{u}^\top\bm{y},
~\forall~\bm{x}\in\mathbb{R}^{n}
\end{equation}
where $\bm{x}$ is the state, $\bm{u}\in
\mathbb{U}\subseteq\mathbb{R}^m$ the input, $\bm{y}\in
\mathbb{R}^{m}$ the output, the drift term
$\bm{f}:\mathbb{R}^{n}\times\mathbb{U} \mapsto \mathbb{R}^{n}$, the
diffusion term $\bm{h}:\mathbb{R}^{n}\times\mathbb{U}\mapsto
\mathbb{R}^{n\times r}$ and
$\bm{s}:\mathbb{R}^{n}\times\mathbb{U}\mapsto \mathbb{R}^{m}$ all
satisfy the condition of local Lipschitz continuity, and $t$, $\boldsymbol{\omega}$ share the same meaning with those in Eq.
(\ref{StochasticEquation}). The nonnegative real function
$S(\bm{x})$ is called the storage function, the state where
$S(\bm{x})=0$ is the stochastic passive state and the inner product
$\bm{u}^\top\bm{y}$ is called the supply rate.
\end{definition}

\begin{result}[Stochastic Passivity Theorem]The negative feedback
connection of two stochastic passive systems is stochastically
passive.
\end{result}
\begin{proof}
{Let $1$ and $2$ in the form of subscripts represent
these two stochastic passive systems, respectively, then we have
$$\mathcal{L}\big[{S_1(\bm{x}_1)}\big]\leq \bm{u}_1^\top\bm{y}_1~~~\mathrm{and}~~~\mathcal{L}\big[{S_2(\bm{x}_2)}\big]\leq \bm{u}_2^\top\bm{y}_2$$
Define the storage function of their negative feedback connection by
$$S(\bm{x})=S_1(\bm{x}_1)+S_2(\bm{x}_2)$$
and note the fact that
$$\bm{x}=(\bm{x}_1^\top,\bm{x}_2^\top)^\top,~\bm{y}=\bm{y}_1=\bm{u}_2,~\bm{u}=\bm{u}_1+\bm{y}_2$$
then we get
$$\mathcal{L}\big[{S(\bm{x})}\big]=\mathcal{L}\big[{S_1(\bm{x}_1)}\big]+\mathcal{L}\big[{S_2(\bm{x}_2)}\big]\leq (\bm{u}-\bm{y}_2)^\top\bm{y}_1+\bm{u}_2^\top\bm{y}_2=\bm{u}^\top\bm{y}$$
Therefore, the result is true. }
\end{proof}

\begin{result}[Stochastic Passivity and Stability in Probability]
{A stochastic passive system with a positive definite
storage function is stable in probability if a stochastic passive
controller with a positive definite storage function is connected in
negative feedback. }
\end{result}
\begin{proof}
Based on \textit{Result} 1, the whole negative feedback connection
is stochastically passive. As long as the input of the stochastic
passive system (labeled by the subscript ``$1$") is manipulated
according to $\bm{u}_1=\mathcal{C}(\bm{y}_1,\bm{x}_2)$, then
$\bm{u}=\mathbbold{0}_m$, which means
$\mathcal{L}\big[{S(\bm{x})}\big]\leq 0$. Here, the operator
$\mathcal{C}(\cdot)$ is the stochastic passive controller (labeled
by the subscript ``$2$") defined by
$\bm{y}_2=-\mathcal{C}(\bm{u}_2,\bm{x}_2)$. The stability in
probability of $\bm{x}(t)$ is immediately from \textit{Theorem} $1$,
so is that of $\bm{x}_1(t)$.
\end{proof}

\begin{rem}
{Deterministic passive systems are a kind of special
cases of stochastic passive systems. Therefore, the frequently-used
passive controllers \cite{Schaft00}, such as PID Controller, Model
predictive Controller, etc., can all serve for stabilizing the
stochastic passive systems in probability. }
\end{rem}

\section{Loss of stochastic passivity and Problem setting}
{This section contributes to elaborating that stochastic passivity
will vanish either in some stochastic systems or when some control
problems are addressed, and further to formulating the problem of
interest.}
\subsection{Loss of stochastic passivity} {As can be known
from \textit{Definition} \ref{Defstochasticpassivity}, {a key point to
capture stochastic passivity lies in finding a storage function.} We
will derive the necessary condition for stochastic passivity in the
following, and then get the sufficient condition to say the loss of
stochastic passivity. For this purpose, we go back to the stochastic
differential equation of Eq. (\ref{StochasticEquation}). }


\begin{thm}\label{Vanish1}If a stochastic differential equation
given by Eq. (\ref{StochasticEquation}) has a global solution, then
it must be not stable in probability at those states that result in
the nonzero diffusion term.
\end{thm}

\begin{proof}
{Let the set of all states result in the nonzero
diffusion term be given by
$$\mathbb{H}_{\neq 0}:=\{\bm{x}^\ddag\in \mathbb{R}^n|\bm{h}(\bm{x}^\ddag)\neq\mathbbold{0}_{n\times r}\}$$
For any $\bm{x}^\ddag\in\mathbb{H}_{\neq 0}$ and $\gamma\textgreater
0$, it is expected that
$$\lim_{\bm{x}(0)\to
\bm{x}^\ddag}\mathrm{P}\big(\sup{\parallel\bm{x}(t)-\bm{x}^\ddag\parallel_2}<\gamma\big)=1$$
must be not true. Towards this purpose, we assume
$\bm{x}^\ddag=\mathbbold{0}_{n}$ for simplicity but without loss of
generality (which means $\bm{h}(\mathbbold{0}_{n})\neq
\mathbbold{0}_{n\times r}$), and further construct a real-valued
function $\tilde{U}:\mathbb{R}\mapsto \mathbb{R}$ in the form of
\begin{equation*}
\tilde{U}(x)= \left\{
\begin{array}{ll}
x^{2}  & 0\leq |x|\leq\frac{1}{2}\\
-\frac{2}{3}x^{3}+2x^{2}-\frac{1}{2}x+\frac{1}{2} & \frac{1}{2}\leq|x|\leq\frac{3}{2}\\
\frac{1}{3}x^{3}-\frac{5}{2}x^2+\frac{25}{4}x-\frac{79}{24} & \frac{3}{2}\leq|x|\leq\frac{5}{2}\\
\frac{23}{12} & \frac{5}{2}\leq|x|
\end{array}
\right..
\end{equation*}
Based on this function, a positive definite, twice continuously
differentiable and bounded real function mapping $\mathbb{R}^n$ to
$\mathbb{R}$ is defined by
$$U(\bm{x})=\frac{12}{23}\times\tilde{U}(\|\bm{x}\|_{2})$$
Clearly, $U(\bm{x})\in [0,1]$ and, moreover,
$U(\bm{x})=\frac{12}{23}\bm{x}^{\top}\bm{x}$ in the
$\frac{1}{2}$-neighborhood of $\mathbbold{0}_{n}$.}

{In order to finish the proof, we impose the
infinitesimal generator $\mathcal{L}[\cdot]$ on $U(\bm{x})$, and are
only concerned about the result at $\bm{x}=\mathbbold{0}_{n}$. From
Eq. (\ref{InfinitesimalGenerator}), we have
\begin{equation*}
\mathcal{L}[U](\mathbbold{0}_{n})=\frac{12}{23}\text{tr}\big\{\bm{h}(\mathbbold{0}_{n})\bm{h}^\top(\mathbbold{0}_{n})\big\}>0.
\end{equation*}
On the other hand, from the definition of $\mathcal{L}[\cdot]$
\cite{Khasminskii11} we get
\begin{equation*}
\mathcal{L}[U](\mathbbold{0}_{n}) =\lim_{t\to 0}
\frac{\text{E}^{\mathbbold{0}_{n}}\Big[U\big(\bm{x}(t)\big)\Big]-U(\mathbbold{0}_{n})}{t}
\end{equation*}
where $\mathbbold{0}_{n}$ appearing in
$\text{E}^{\mathbbold{0}_{n}}\Big[U\big(\bm{x}(t)\big)\Big]$
indicates that the initial condition is
$\bm{x}(0)=\mathbbold{0}_{n}$. Since the stochastic differential
equation (\ref{StochasticEquation}) has a global solution, there
exist a time $\tau \textgreater 0$ and a constant $c \textgreater 0$
so that
{$\text{E}^{\mathbbold{0}_{n}}\Big[U\big(\bm{x}(\tau)\big)\Big]=c\tau\textgreater
0$}. Also, since
\begin{eqnarray*}
\text{E}^{\mathbbold{0}_{n}}\Big[U\big(\bm{x}(\tau)\big)\Big]&=&
\text{E}^{\mathbbold{0}_{n}}\Big[~U\big(\bm{x}(\tau)\big)~\Big|~U\big(\bm{x}(\tau)\big)<{\epsilon^2}~\Big] \text{P}\Big(U\big(\bm{x}(\tau)\big)<{\epsilon^2}~\big|~\bm{x}(0)=\mathbbold{0}_{n}\Big) \notag \\
&+& \text{E}^{\mathbbold{0}_{n}}\Big[~U\big(\bm{x}(\tau)\big)~\Big|
~U\big(\bm{x}(\tau)\big)\geq{\epsilon^2}~\Big]
\text{P}\Big(U\big(\bm{x}(\tau)\big)\geq{\epsilon^2}~\big|~\bm{x}(0)=\mathbbold{0}_{n}\Big)\notag
\end{eqnarray*}
and
$$\text{E}^{\mathbbold{0}_{n}}\Big[~U\big(\bm{x}(\tau)\big)~\Big|~U\big(\bm{x}(\tau)\big)<{\epsilon^2}~\Big]\textless
\epsilon^2,~\text{P}\Big(U\big(\bm{x}(\tau)\big)<{\epsilon^2}~\big|~\bm{x}(0)=\mathbbold{0}_{n}\Big)\leq
1$$ together with the fact that
$$U\big(\bm{x}(\tau)\big)\in [0,1]\Rightarrow\text{E}^{\mathbbold{0}_{n}}\Big[~U\big(\bm{x}(\tau)\big)~\Big|
~U\big(\bm{x}(\tau)\big)\geq{\epsilon^2}~\Big]\textless 1$$ where
$\epsilon$ is any positive number, we have
$$\text{P}\left(~U\big(\bm{x}(\tau)\big)\geq{\epsilon^2}~\big|~\bm{x}(0)=\mathbbold{0}_{n}~\right)\geq
c\tau-\epsilon^2$$ We set $\epsilon$ to be sufficiently small so
that
\begin{equation*}{\label{ProbabilityEsitimation}}
\text{P}\big(~\|\bm{x}(\tau)\|_{2}\geq{\gamma}~\big|~\bm{x}(0)=\mathbbold{0}_{n}~\big)\geq
c\tau-\epsilon^2>0
\end{equation*}
where $\gamma=\sqrt{\frac{23}{12}}\epsilon$. }

From the definition, the leftmost term in the above inequality can
be calculated by
\begin{align*}
&\text{P}\big(~\|\bm{x}(\tau)\|_{2}\geq{\gamma}~\big|~\bm{x}(0)=\mathbbold{0}_{n}~\big) = \\
&\int_{\|\bm{y}\|_2=\delta}
\text{P}\big(~\bm{x}(\tau_{\delta})=\dd\bm{y}~\big|~\bm{x}(0)=\mathbbold{0}_{n}~\big)
\text{P}\big(~\|\bm{x}(\tau))\|_{2}\geq{\gamma} ~\big|~
\bm{x}(\tau_{\delta})=\bm{y} ~\big)
\end{align*}
where
$$\tau_{\delta}=\tau\wedge\inf\big\{~t~|~\|\bm{x}(t)\|_{2}=\delta<\gamma~\big\}$$
is a stopping time. Then there exists at least one point, denoted
by $\bm{y}_\delta$, on the surface of the ball
$\|\bm{x}(t)\|_2=\delta$ such that
$$\text{P}\big(~\|\bm{x}(\tau))\|_{2}\geq{\gamma} ~\big|~
\bm{x}(\tau_{\delta})=\bm{y}_\delta ~\big)\geq c\tau-\epsilon^2.$$
Note that Eq. (\ref{StochasticEquation}) is autonomous, therefore
$$\text{P}\left(\sup_{t\in[0,\infty)}\|\bm{x}(t)\|_2\geq\gamma ~\Big|~ \bm{x}(0)=\bm{y}_\delta \right)
\geq \text{P}\Big(\|\bm{x}(\tau))\|_{2}\geq{\gamma} \Big|~
\bm{x}(\tau_{\delta})=\bm{y}_\delta \Big) \geq c\tau-\epsilon^2
$$
Namely, for $\forall~ \delta\leq\gamma$ there always exist a
$\bm{y}_\delta$ to make the above inequality be true. Clearly, the
above inequality suggests that Eq. (\ref{StochasticEquation}) must
be not stable in probability at $\bm{x}=\mathbbold{0}_{n}$.
\end{proof}

{It is straightforward to write the inverse negative
proposition of \textit{Theorem} \ref{Vanish1} as a corollary.
\begin{cor}\label{Vanish2}
For a stochastic differential equation in the form of
(\ref{StochasticEquation}) with a global solution, if it is stable
in probability at a desired state $\bm{x}^\dag$ (may be not the
equilibrium $\bm{x}^*$), then $\bm{x}^\dag$ must belong to
$\mathbb{H}_{=0}$ which is defined by
\begin{equation}
\mathbb{H}_{=0}:=\{\bm{x}^\dag\in
\mathbb{R}^n|\bm{h}(\bm{x}^\dag)=\mathbbold{0}_{n\times r}\}
\end{equation}
\end{cor}}

{Note that the above result depends on the condition
that the stochastic differential equation (\ref{StochasticEquation})
has a global solution. However, under the condition of local
Lipschitz continuity, Eq. (\ref{StochasticEquation}) has a unique
solution only before explosion time. Based on this result, we will
reveal that there is no explosion for some stochastic passive
systems, so it must have a global solution. To this task, attention
is turned to the Non-explosion condition of a stochastic
differential equation proposed by Narita \cite{narita82}.}

\begin{lem}[Non-explosion Condition \cite{narita82}]
Given a stochastic differential equation represented by Eq.
(\ref{StochasticEquation}), if for $\forall~ T\textgreater 0$, there
exist two positive numbers $c_{T}\textgreater 0$ and
$R_{T}\textgreater 0$, and a scalar function
$U_{T}\in\mathscr{C}^{2}\big([0,T]\times
\mathbb{R}^{n};\mathbb{R}\big)$ such that
\begin{equation}\label{Non-explosion1}
\mathcal{L}[{U_T(t,\bm{x})}]\leq c_T
\end{equation}
holds for all $t\leq T$ and $\|\bm{x}\|_{2}\geq R_T$, and moreover,
\begin{equation}\label{Non-explosion2}
\lim_{\|\bm{x}\|_{2}\to\infty}\inf_{0\leq t\leq
T}U_T(t,\bm{x})=\infty
\end{equation}
then the solutions of Eq. (\ref{StochasticEquation}) are of
non-explosion, i.e., the explosion time beginning at any
$t_0\textgreater 0$ and $\bm{x}_0\in\mathbb{R}^{n}$, denoted by
$t_e(t_{0},\bm{x}_{0})$, satisfying
$$\mathrm{P}\big(t_e(t_{0},\bm{x}_{0})=\infty\big)=1$$
\end{lem}
{~~~~~In the following, that \textit{Lemma} $1$ is
applied to a stochastic passive system yields}
\begin{pro}
For a stochastic differential system $\Sigma_S$ governed by Eq.
(\ref{StochasticSystem}), if there exists a {radially unbounded} Lyapuonv function
so that $\Sigma_S$ is stochastically passive, then the unforced
version of Eq. (\ref{StochasticSystem}) has a global solution.
\end{pro}
\begin{proof}
{Assume $V(\bm{x})$ to be the {radially unbounded} Lyapuonv
function that suggests $\Sigma_S$ to be stochastically passive, then
by designating the zero controller to $\Sigma_S$, i.e.,
$\bm{u}=\mathbbold{0}_m$, we have {$\mathcal{L}[V(\bm{x})]\leq 0$}. It
is naturally to observe that $V(\bm{x})$ satisfies Eq.
(\ref{Non-explosion2}). Note that the state evolution of this
unforced version of $\Sigma_S$ is just the same as Eq.
(\ref{StochasticEquation}). Hence, the solutions of $\Sigma_S$ are
of non-explosion based on \textit{Lemma} $1$. Namely, Eq.
(\ref{StochasticSystem}) has a global solution. }
\end{proof}

{From \textit{Proposition} $1$, one can know that
some stochastic passive systems must have a global solution without
force. Combining this result with \textit{Corollary} $1$, we get the
necessary condition for saying $\Sigma_S$ to be of stochastic
passivity, which is expressed as follows.}

\begin{thm}[Necessary Condition for Stochastic Passivity] If there exists
a {radially unbounded} Lyapunov function that can render a stochastic differential
system $\Sigma_S$ described by Eq. (\ref{StochasticSystem}) to be
stochastically passive, then the unforced diffusion term must vanish
at the stochastic passive state.
\end{thm}
\begin{proof}
From \textit{Theorem} $1$ of Stochastic Lyapunov theorem, $\Sigma_S$
is stable in probability at the stochastic passive state with the
zero controller. This together with \textit{Proposition} $1$ and
\textit{Corollary} $1$ yields the result to be true.
\end{proof}

We further express the inverse negative proposition of
\textit{Theorem} $3$ to get the sufficient condition for loss of
stochastic passivity.

\begin{cor}[Sufficient Condition for Loss of Stochastic Passivity]\label{Vanish3}
If the unforced diffusion term $\bm{h}(\bm{x},\mathbbold{0}_{m})$ in
a stochastic differential system $\Sigma_S$ in the form of
(\ref{StochasticSystem}) does not vanish at any state $\bm{x}\in
\mathbb{R}^n$, then there does not exist any {radially unbounded} Lyapunov
function to ensure $\Sigma_S$ to be stochastically passive.
\end{cor}

\begin{rem}
\textit{Corollary} \ref{Vanish3} implies that stochastic passivity
will lose when the desired state $\bm{x}^\dag$ makes
$\bm{h}(\bm{x}^\dag,\mathbbold{0}_{m})\neq \mathbbold{0}_{n\times
r}$ and the storage function is expected to be a {radially unbounded} Lyapunov
function, so it is impossible for one to use stochastic passivity
theory, and further, stochastic Lyapunov theorem to analyze the
globally asymptotical stability of $\Sigma_S$ at $\bm{x}^\dag$ in
the sense of probability.
\end{rem}

\subsection{Problem setting}
{The above analysis reveals that when
$\bm{h}(\bm{x}^\dag,\mathbbold{0}_{m})\neq \mathbbold{0}_{n\times
r}$, stochastic passivity will fail to capture the globally
asymptotical stability (in the probability sense) of a stochastic
differential system $\Sigma_S$ at the desired state $\bm{x}^\dag$,
often set as the equilibrium state $\bm{x}^*$ (if exists) in many
control problems. In fact, the nonzero diffusion term is frequently
encountered in many real stochastic systems, such as chemical
reaction networks, tracking systems, etc. One case is that the noise
is persistent in quite a few stochastic systems, which means that
for $\forall~ \bm{x}\in \mathbb{R}^n$ and $\forall~\bm{u}\in
\mathbb{U},~\bm{h}(\bm{x},\bm{u})\neq \mathbbold{0}_{n\times r}$ and
thus $\bm{x}^*$ does not exist at all; the other case is that some
special control purposes are served for, such as the desired state
$\bm{x}^\dag$ being not $\bm{x}^*$ so that
$\bm{h}(\bm{x}^\dag,\mathbbold{0}_{m})\neq \mathbbold{0}_{n\times
r}$, even if $\bm{x}^*$ exists. }

Apparently, the nonzero diffusion term in real stochastic systems
restricts greatly the applications of stochastic passivity theory, a
powerful tool for stabilization. However, what is even worse is that
it may lead to the system under consideration being not stable at
all in probability at the desired state, as stated in
\textit{Theorem} $2$. These two awkward situations motivate us to
find a new solution for stabilizing those stochastic systems with
nonzero diffusion term at the desired state. On the one hand, it is
impossible to stabilize some stochastic systems at any state in
probability, on the other hand, the excellent performance of
stochastic passivity is hoped to be used. Thus, we take a hack at
the next best way to address the current control problem, including
{seeking the convergence in distribution and ergodicity instead
of the convergence in probability}, and finding the stochastic
passivity behavior only outside a certain neighborhood of the
desired state instead of in the whole state domain.

\section{Stochastic weak passivity theory}
{The objective in this section is to present the theory of
stochastic weak passivity with which some stochastic systems with
nonzero diffusion term can be analyzed concerning the convergence of
the transition measure {and ergodicity}. This theoretical
framework includes some basic concepts related to stochastic weak
passivity, properties of invariant measure, and results for
stabilization which are parallel to those appearing in the
stochastic passivity theory.}
\subsection{Basic concepts} We firstly give the
definitions of convergence in distribution and {of ergodicity}.

\begin{definition}[Convergence in distribution and Ergodicity]
Assume a stochastic differential equation described by Eq.
(\ref{StochasticEquation}) to have an invariant measure $\pi$. If
there exists a subset of $\mathbb{R}^n$, denoted by
$\mathbb{R}^n_\pi$, such that for any Borel subset $\mathbb{A}$ with
zero $\pi$-measure boundary the equation
\begin{equation}\label{StabilityMeasure}
\lim_{t\to\infty}\mathcal{P}(t,\bm{x}(0),\mathbb{A})=\pi(\mathbb{A}),~
\forall~\bm{x}(0)\in \mathbb{R}^n_\pi
\end{equation}
is true, then the stochastic process is said to be locally
convergent in distribution. If $\mathbb{R}^n_\pi=\mathbb{R}^n$, then
the convergence in distribution is globally.

{If for any Borel subset $\mathbb{B}$ the state $\bm{x}(t)$ satisfies}
\begin{equation}{\label{Eq ergodiacity}}
\lim_{T\to\infty} \frac{1}{T}\int_{0}^{T}
\mathbbold{1}_{\left\{\bm{x}(t)\in \mathbb{B}\right\}} dt
=\pi(\mathbb{B}) ,~~~\mathrm{a.s.}~~ \forall~\bm{x}(0)\in \mathbb{R}^n_\pi
\end{equation}
{where ``$\mathrm{a.s.}$" represents ``almost surely" and}
\begin{equation*}
\mathbbold{1}_{\left\{\bm{x}(t)\in \mathbb{B}\right\}}= \left\{
\begin{array}{ll}
1  & \bm{x}(t)\in \mathbb{B},\\
0 & \bm{x}(t)\notin \mathbb{B}
\end{array}
\right.
\end{equation*}
{then the stochastic process is said to be locally ergodic.
Especially, when $\mathbb{R}^n_\pi=\mathbb{R}^n$, the ergodicity is
global.}
\end{definition}

Here, analogous to the definition of stability in probability, we
also distinguish the local and global notations to emphasize the
importance of the initial condition.

\begin{rem} {In the control view point, the convergence of the
transition measure and ergodicity both describe certain senses of stable
behaviors for stochastic systems. The former means that the
distribution of the state will converge to an invariant measure as
time goes infinite. Therefore, as long as the invariant measure is
shaped to fasten on a small region around the desired state, then
the state will evolve within this region with a large probability,
i.e., not to deviate from the desired point too far with a large
probability. The latter implies that the state evolution almost
always take place within the mentioned region. Even if the
trajectory sometimes run from the region, it will come back into the
region immediately.}
\end{rem}

{Clearly, the convergence of the transition measure and
ergodicity reveal that the state of a stochastic system almost
always evolves near the desired state if the invariant measure is
assigned properly. We define this behavior as stochastic asymptotic
weak stability.}

\begin{definition}[Stochastic Asymptotic Weak Stability] A stochastic differential
equation of Eq. (\ref{StochasticEquation}) is of
local$\backslash$global stochastic asymptotic weak stability if its
distribution locally$\backslash$globally converges to an invariant
measure {and its process is of local$\backslash$global
ergodicity}.
\end{definition}

Next, we define the stochastic weak passivity that serves for
stabilizing a stochastic differential system in weak sense. Note
that the loss of stochastic passivity mainly originates from the
nonzero diffusion term at the desired state, which further results
in some unexpected behaviors appearing around it. Thus, a naive idea
is to give up the stochastic passivity near the desired state, but
only to suggest it outside a neighborhood of the desired state.

\begin{definition}[Stochastic Weak Passivity]{\label{DefStochasticWeakPassivity}}
A stochastic differential system $\Sigma_S$, as described by Eq.
(\ref{StochasticSystem}), is said to be of stochastic weak passivity
if there exist a $\mathscr{C}^2(\mathbb{R}^n;\mathbb{R}_{\geq 0})$
function $V(\bm{x})$, i.e., the storage function, such that for
$\forall ~\bm{x}\in \mathbb{R}^n$ and $\|\bm{x}-\bm{x}_R\|_{2}\geq
R$ the following inequality holds
\begin{equation*}
\mathcal{L}[{V(\bm{x})}]\leq \bm{u}^\top\bm{y}
\end{equation*}
where {the state $\bm{x}_R\in \mathbb{R}^n$ is the sole minimum
point for $V(x)$} and $R\geq 0$ is called the stochastic passive
radius.
\end{definition}


Similar to the concept of the strict passivity, we may further
define strict stochastic weak passivity.

\begin{definition}[Strict Stochastic Weak Passivity]{\label{DefStrictStochasticWeakPassivity}}
Consider a stochastic weak passive system. Suppose that there exists
a positive constant $\delta$ such that for $\forall ~\bm{x}\in
\mathbb{R}^n$ and $\|\bm{x}-\bm{x}_R\|_{2}\geq R$
\begin{equation*}
\mathcal{L}[{V(\bm{x})}]\leq
\bm{u}^\top\bm{y}-\delta\|\boldsymbol{\xi}\|_2^{2}
\end{equation*}
The system is
\begin{itemize}
  \item strictly state stochastic weak passive if $\boldsymbol{\xi}=\bm{x}-\bm{x}_R$.
  \item strictly input stochastic weak passive if $\boldsymbol{\xi}=\bm{u}$.
  \item strictly output stochastic weak passive if $\boldsymbol{\xi}=\bm{y}$.
\end{itemize}
\end{definition}

\subsection{Properties of invariant measure}{Definition
$4.2$ reveals that the stochastic asymptotic weak stability is
concerned with the convergence in distribution of the state {and
ergodic behavior}. However, for a stochastic system, unlike its
equilibrium it is not quite obvious to know something about its
invariant measure, such as the existence, uniqueness, etc. We
separate this subsection to analyze the properties of the invariant
measure of the stochastic differential equation under consideration.
}

{In fact, it is not a new research issue to analyze the properties
of the invariant measure of a stochastic system
\cite{Khasminskii60,Stettner94}. A sufficient condition to say it
convergent in distribution was reported as follows.}

\begin{thm}[cf. \cite{Stettner94}]
If a right Markov process on $\mathbb{R}^{n}$ is strongly Feller,
i.e., $\forall ~t\textgreater 0$ the transition semigroup
$\mathcal{P}(t,\cdot,\cdot)$ transforms bounded Borel functions into
$\mathscr{C}(\mathbb{R}^{n})$, and moreover $\mathcal{P}$ is
irreducible, i.e., $\forall ~t\textgreater 0,~\forall~
\bm{x}\in\mathbb{R}^{n}$ and any open set $\mathcal{O}\neq
\varnothing$ there is $\mathcal{P}(t,\bm{x},\mathcal{O})>0$, then
any probability measure converges to the invariant measure (if
exists). Moreover the invariant measure (if exists) is equivalent to
each transition measure $\mathcal{P}(t,\bm{x},\cdot)$, $t>0$,
$\bm{x}\in\mathbb{R}^{n}$.
\end{thm}

{This theorem provides a solution to capture the convergence in
distribution for a right Markov process. However, it is not easy to
verify the conditions of ``strongly Feller" and ``irreducible" in
practical applications. As an alternative, Khasminskii
\cite{Khasminskii11} proposed a more practical way to say a
stochastic system to be convergent in distribution, which works if a
Markov process is ``mix sufficiently well" in an open domain
$\mathcal{O}$ and the recurrent time is finite (cf.
\textit{Theorems} $4.1$, $4.3$ and \textit{Corollary} $4.4$ in
\cite{Khasminskii11}). Here, we will combine this practical way with
Zakai's work \cite{zakai69}, and give a Lyapunov criterion to say
stochastic asymptotic weak stability. However, the drift and
diffusion terms of the stochastic system are set to have local
Lipschitz continuity instead of global Lipschitz continuity in
\cite{zakai69}. }

\begin{lem}[Finite Mean Recurrent Time \cite{zakai69}]{\label{lem finite mean recurrent time}}
For a
stochastic differential equation (\ref{StochasticEquation}) having a
global solution $\bm{x}(t)$, if there exist a function $V(\bm{x})\in
\mathscr{C}^{2}(\mathbb{R}^{n};\mathbb{R}_{\geq 0})$, a state
$\bm{x}_{\tilde{R}}$, and two positive numbers $\tilde{R}$ and $k$
such that
\begin{equation}\label{WeakStabilityLya}
\mathcal{L} [V(\bm{x})]<-k, ~\forall~ \|\bm{x}(t)-
\bm{x}_{\tilde{R}}\|_{2}\geq \tilde{R}
\end{equation}
then for all $\bm{x}(0)\in\mathbb{R}^n$ the first passage time from
$\bm{x}(0)$ to the sphere $\|\bm{x}(t)-\bm{x}_{\tilde{R}}\|_{2}\leq
\tilde{R}$, denoted by $\tau$, satisfies
\begin{equation}{\label{IneqFMRT}}
\mathrm{E}\big[\tau|\bm{x}(0)\big]\leq\frac{V\big(\bm{x}(0)\big)}{k}
\end{equation}
\end{lem}

\begin{proof}
At the time of $t\wedge\tau$, by Dynkin's formula we have
\begin{eqnarray*}
\text{E}\Big[V\big(\bm{x}(t\wedge\tau)\big)~\Big|~\bm{x}(0)\Big]&=& V\big(\bm{x}(0)\big)+\text{E}\left[\int_{0}^{t\wedge\tau}\mathcal{L}\big[V\big(\bm{x}(s)\big)\big]ds\right]\\
&\leq&
V\big(\bm{x}(0)\big)-\text{E}\Big[\int_{0}^{t\wedge\tau}kds\Big]
\end{eqnarray*}
Note that $V(\bm{x})\geq 0$, so
$\mathrm{E}\big[t\wedge\tau\big|\bm{x}(0)\big]\leq\frac{V\left(\bm{x}(0)\right)}{k}$.
The inequality (\ref{IneqFMRT}) naturally holds due to the monotone
convergence.
\end{proof}

\begin{thm}{\label{WeakStability}}
For a stochastic equation in the form of (\ref{StochasticEquation}),
if there exists a nonnegative function $V(\bm{x})\in
\mathscr{C}^{2}(\mathbb{R}^{n};\mathbb{R}_{\geq 0})$ satisfying the
following conditions:
\begin{itemize}
  \item $\lim_{\|\bm{x}\|_{2}\to\infty} V(\bm{x})=\infty$;
  \item $\exists~\bm{x}_{\tilde{R}}\in\mathbb{R}^n,~ \tilde{R}\textgreater 0~and~k\textgreater 0,~if~\|\bm{x}-\bm{x}_{\tilde{R}}\|_{2}\geq \tilde{R},~then~\mathcal{L}
  [V(\bm{x})]<-k$;
  \item $\exists~\epsilon\textgreater 0$, if $\|\bm{x}-\bm{x}_{\tilde{R}}\|_2< \tilde{R}+\epsilon$, then $rank(\bm{h}(\bm{x})\bm{h}^{\top}(\bm{x}))=n$.
\end{itemize}
then there is a unique finite invariant measure $\pi$ such that for
any Borel subset $\mathbb{A}$ with zero $\pi$-measure boundary
\begin{equation}{\label{TransitionMeasureConvergence}}
\lim_{t\to\infty}\mathcal{P}(t,\bm{x}(0),\mathbb{A})=\pi(\mathbb{A}),~\forall~
\bm{x}(0)\in\mathbb{R}^{n}
\end{equation}
{and for any Borel subset $\mathbb{B}$}
\begin{equation}{\label{Eq global ergodicity}}
\lim_{T\to\infty} \frac{1}{T}\int_{0}^{T} \mathbbold{1}_{\left\{\bm{x}(t)\in \mathbb{B}\right\}} \dd t =\pi(\mathbb{B}), ~~~\mathrm{a.s.}~~~\forall~
\bm{x}(0)\in\mathbb{R}^{n}
\end{equation}
i.e., Eq. (\ref{StochasticEquation}) being globally asymptotical
stable in weak sense.
\end{thm}

\begin{proof}
According to \textit{Lemmas} $1$ and $2$, the first two conditions
could suggest that Eq. (\ref{StochasticEquation}) has a unique
global solution and for any initial state $\tilde{\bm{x}}$
satisfying $\tilde{\bm{x}}\in
\{\tilde{\bm{x}}|\|\tilde{\bm{x}}-\bm{x}_{\tilde{R}}\|_2\geq
\tilde{R}\}$, we have
$$\mathrm{E}\big[\tau\big|\tilde{\bm{x}}\big]\leq\frac{V\left(\tilde{\bm{x}}\right)}{k}$$
Hence, for any compact subset $\mathcal{K}\in \mathbb{R}^{n}$ we get
\begin{equation*}
\sup_{\tilde{\bm{x}}\in\mathcal{K}}
\text{E}[\tau|\tilde{\bm{x}}]\leq
\sup_{\tilde{\bm{x}}\in\mathcal{K}}\frac{V\big(\tilde{\bm{x}}\big)}{k}<\infty.
\end{equation*}
{Further based on the strong maximum principle for solutions of
elliptic equations, the third condition implies the system
\mbox{(\ref{StochasticEquation})} to be irreducible (cf. \textit{Lemma
4.1} in \mbox{\cite{Khasminskii11}}), which combining the above
inequality suggests that an ergodic Markov chain can be induced for
this stochastic process by constructing a circle. The ergoic
property of the Markov chain will ensure that there exists a sole
invariant measure to which the transition measure converges (cf.
\textit{Theorems} $4.1$ and $4.3$, and \textit{Corollary} $4.4$ in
\mbox{\cite{Khasminskii11}}), and the ergodicity of the system under
consideration is true (cf. \textit{Theorem} $4.2$ in
\mbox{\cite{Khasminskii11}}). Namely, Eqs.
\mbox{(\ref{TransitionMeasureConvergence})} and \mbox{(\ref{Eq global
ergodicity})} hold}.
\end{proof}

The theorem provides a Lyapunov function based method to address the
issues of the existence and uniqueness of the invariant measure
together with the convergence of the transition probability measure
{and ergodicity}
for a stochastic differential equation, so it can be associated with
the Lyapunov stability theory conveniently.

\begin{rem}
{There are two differences between the above theorem
and the corresponding result in \cite{zakai69}. One is that the
non-singularity of $\bm{h}(\bm{x})\bm{h}^{\top}(\bm{x})$ is not
necessary in the whole state space but only holds in an open ball
($\|\bm{x}-\bm{x}_{\tilde{R}}\|_2< \tilde{R}+\epsilon$). The latter
is believed to be achieved more easily in practice. The other is
that the storage function must be {radially unbounded} here. In fact, this is not
a necessary condition, which can be removed if the drift term and
diffusion term in the stochastic equation are assumed to be globally
Lipschitz continuous. }
\end{rem}

{For a stochastic asymptotic weak stable system, to ensure the state to evolve
within a small region around the desired point, the invariant measure needs to be assignable or at least partially shaped by the control to fasten on this region. In the sequel, we will prove that the invariant measure can be shaped purposefully by controlling the change rates of the nonnegative function $V(\bm{x})$ and the radius $\tilde{R}$ of the ball $\|\bm{x}-\bm{x}_{\tilde{R}}\|_{2}\geq \tilde{R}$.}

\begin{lem}{\label{lem stop time estimation1}}
{For a stochastic differential equation \mbox{(\ref{StochasticEquation})} admitting a
global solution $\bm{x}(t)$, if $\exists$ $V(\bm{x})\in
\mathscr{C}^{2}(\mathbb{R}^{n};\mathbb{R}_{\geq 0})$ and $k,C\in\mathbb{R}_{\textgreater 0}$ such that}
\begin{equation*}
\mathcal{L} V(\bm{x})\leq \left\{
\begin{array}{ll}
-k  & \forall~ \bm{x}\in\left\{\bm{x}~|~V(\bm{x})\geq V_{1}\right\},\\
C & \forall~ \bm{x}\in\mathbb{R}^{n}
\end{array}
\right.
\end{equation*}
{and}
\begin{equation}{\label{eq infinite loops}}
\text{P}(\tau_{2i}-\tau_{2i-1}=\infty)=0 \quad \forall~ i\in\mathbb{Z}_{\textgreater 0}
\end{equation}
{then for any $i\geq 1,~i\in\mathbb{Z}_{\textgreater 0}$ we have}
$$~~~~~~~~~~~\mathrm{(1)~ E}\left[\tau_{2i}-\tau_{2i-1}\right]\geq \frac{\left(V_{2}-V_{1}\right)^{2}}{2 C V_{2}};
~\mathrm{(2)~ E}\left[\tau_{2i-1}-\tau_{2i-2}\right]\leq \frac{V_{2}-V_{1}}{k};~~\mathrm{and}$$
$$\mathrm{(3)~ E}\left[\liminf_{T\to\infty}~
\frac{1}{T}\int_{0}^{T}\mathbbold{1}_{\left\{V\left(\bm{x}(t)\right)\geq V_{2}\right\}} \dd t\right]
\leq \frac{2 C V_{2}}{2 C V_{2}+k\left(V_{2}-V_{1}\right)}$$
{where $V_1,V_2\in\mathbb{R}_{\textgreater 0}$ satisfying $V_{2}>V_{1}$, $\tau_{2i-2}$ represents the first time at which the state hits the region $\{V(x)\geq V_{2}\}$ after $\tau_{2i-3}$,
$\tau_{2i-1}$ is the first time at which the trajectory reaches
the surface of $\{V(x)\leq V_{1}\}$ after $\tau_{2i-2}$, and $\tau_{-1}$ means the initial time.}
\end{lem}

\begin{proof}
{(1) According to Dynkin's formula, for any $i\geq 1,~i\in\mathbb{Z}_{\textgreater 0}$}
\begin{eqnarray*}
\text{E} \left[V\left(\bm{x}(t\wedge\tau_{2i})\right)\right]
&=&V\left(\bm{x}(\tau_{2i-1})\right)
+\text{E} \left[\int_{\tau_{2i-1}}^{t\wedge\tau_{2i}} \mathcal{L}V\left(\bm{x}(s)\right) ds\right] \\
&\leq&V\left(\bm{x}(\tau_{2i-1})\right)
+\text{E}\left[ \int_{\tau_{2i-1}}^{t\wedge\tau_{2i}} C ds\right]\\
&\leq&V_{1}
+C\left(t-{\tau_{2i-1}}\right)
\end{eqnarray*}
Also, since
\begin{eqnarray*}
\text{E} \left[V\left(\bm{x}(t\wedge\tau_{2i})\right)\right]&=&\text{E} \left[V\left(\bm{x}(t)\right)\right]\text{P}\left(\tau_{2i}\textgreater t\right)+\text{E} \left[V\left(\bm{x}(\tau_{2i})\right)\right]\text{P}\left(\tau_{2i}\leq t\right)\geq V_{2} \text{P}\left(\tau_{2i}\leq t\right)\\
\end{eqnarray*}
we have
\begin{equation*}
\text{P}\left(\tau_{2i}\leq t\right)
\leq\frac{\text{E} \left[V\left(\bm{x}(t\wedge\tau_{2i})\right)\right]}{V_{2}}
\leq\frac{V_{1}+C\left(t-{\tau_{2i-1}}\right)}{V_{2}}
\end{equation*}
Therefore, we get
\begin{eqnarray*}
\text{E}\left[\tau_{2i}-\tau_{2i-1}\right]
&=&\int_{0}^{\infty} \text{P}\left(\tau_{2i}-\tau_{2i-1} > s \right) ds
\geq \int_{0}^{\frac{V_{2}-V_{1}}{C}} \text{P}\left(\tau_{2i}> s+\tau_{2i-1} \right) ds\\
&\geq& \int_{0}^{\frac{V_{2}-V_{1}}{C}} 1-\frac{V_{1}+Cs}{V_{2}}ds=\frac{\left(V_{2}-V_{1}\right)^{2}}{2 C V_{2}}
\end{eqnarray*}

{(2) On the other side, for any $i>1$ we have}
\begin{eqnarray*}
\text{E} \left[V\left(\bm{x}(t\wedge\tau_{2i-1})\right)\right]
&=&V\left(\bm{x}(\tau_{2i-2})\right)
+\text{E} \left[\int_{\tau_{2i-2}}^{t\wedge\tau_{2i-1}} \mathcal{L}V\left(\bm{x}(s)\right) ds\right] \\
&\leq&V\left(\bm{x}(\tau_{2i-2})\right)
+\text{E}\left[ \int_{\tau_{2i-2}}^{t\wedge\tau_{2i-1}} -k ds\right]\\
&=&V_{2}
-k\text{E}\left[t\wedge\tau_{2i-1}-{\tau_{2i-2}}\right]
\end{eqnarray*}
i.e., $$\text{E}\left[t\wedge\tau_{2i-1}-{\tau_{2i-2}}\right]\leq\frac{V_{2}-\text{E} \left[V\left(\bm{x}(t\wedge\tau_{2i-2})\right)\right]}{k}\leq\frac{V_{2}-V_{1}}{k}$$
{By monotone convergence theorem, the inequality
$\text{E}\left[\tau_{2i-1}-{\tau_{2i-2}}\right]\leq\frac{V_{2}-V_{1}}{k}$ is true.}

{(3) Based on the results of (1) and (2), we have that for any $j\in\mathbb{Z}_{\textgreater 0}$}
\begin{equation*}
\text{E}\left[\frac{\sum_{i=1}^{j} (\tau_{2i}-\tau_{2i-1})}{\sum_{i=1}^{j} (\tau_{2i+1}-\tau_{2i})}\right]
\geq \frac{k \left(V_{2}-V_{1}\right)}{2 C V_{2}}.
\end{equation*}
{Besides, Eq. \mbox{(\ref{eq infinite loops})} and the result (2) imply there're almost surely infinite many $\tau_{i}$, so the notations ``$\limsup$" and ``$\liminf$" in the following are not in vain.}
{Applying Fatou's lemma yields}
\begin{equation}{\label{eq lower bound}}
\text{E}\left[\limsup_{j\to\infty}\frac{\sum_{i=1}^{j} (\tau_{2i}-\tau_{2i-1})}{\sum_{i=1}^{j} (\tau_{2i+1}-\tau_{2i})}\right]
\geq \frac{k \left(V_{2}-V_{1}\right)}{2 C V_{2}}
\end{equation}

{Let $i(T)=\max\{i~| \tau_{2i} \leq T\}$, utilizing which we have}
\begin{eqnarray*}
\frac{1}{T}\int_{0}^{T}\mathbbold{1}_{\left\{V\left(\bm{x}(t)\right)\geq V_{2}\right\}} dt
&=& \frac
{\sum_{i=0}^{i(T)-1} \int^{\tau_{2i+2}}_{\tau_{2i}}\mathbbold{1}_{\left\{V\left(\bm{x}(t)\right)>V_{2}\right\}} dt
+\int^{T}_{\tau_{2i(T)}}\mathbbold{1}_{\left\{V\left(\bm{x}(t)\right)>V_{2}\right\}} dt}
{T}\\
&\leq& \frac{\tau_{1}-\tau_{0}}{T}\mathbbold{1}_{\{i(T)\geq 1\}} +
       \frac{\sum_{i=1}^{i(T)-1} (\tau_{2i+1}-\tau_{2i})}
        {\tau_{0}+\sum_{i=0}^{i(T)-1} (\tau_{2i+2}-\tau_{2i})+(T-\tau_{2i(T)})}\\
&&     +\frac{(\tau_{2i(T)+1}-\tau_{2i(T)})\mathbbold{1}_{\{T>\tau_{2i(T)+1}\}}
         +(T-\tau_{2i(T)})\mathbbold{1}_{\{T<\tau_{2i(T)+1}\}}}
         {\tau_{0}+\sum_{i=0}^{i(T)-1} (\tau_{2i+2}-\tau_{2i})+(T-\tau_{2i(T)})} \\
&\leq&  \frac{\tau_{1}-\tau_{0}}{T}\mathbbold{1}_{\{i(T)\geq 1\}}+
        \frac{\sum_{i=1}^{i(T)-1} (\tau_{2i+1}-\tau_{2i})+(\tau_{2i(T)+1}-\tau_{2i(T)})}
        {\sum_{i=0}^{i(T)-1} (\tau_{2i+2}-\tau_{2i})+(\tau_{2i(T)+1}-\tau_{2i(T)})}\\
&=&     \frac{\tau_{1}-\tau_{0}}{T}\mathbbold{1}_{\{i(T)\geq 1\}}+
        \frac{\sum_{i=1}^{i(T)} (\tau_{2i+1}-\tau_{2i})}
        {\sum_{i=0}^{i(T)} (\tau_{2i+1}-\tau_{2i})
         +\sum_{i=1}^{i(T)} (\tau_{2i}-\tau_{2i-1})}\\
&\leq&  \frac{\tau_{1}-\tau_{0}}{T}\mathbbold{1}_{\{i(T)\geq 1\}}+
        \frac{1}{1+\frac{\sum_{i=1}^{i(T)} (\tau_{2i}-\tau_{2i-1})}{\sum_{i=1}^{i(T)} (\tau_{2i+1}-\tau_{2i})}}
\end{eqnarray*}
Hence,
\begin{eqnarray*}
\liminf_{T\to\infty}
\frac{1}{T}\int_{0}^{T}\mathbbold{1}_{\left\{V\left(\bm{x}(t)\right)\geq V_{2}\right\}} dt
&\leq&
\liminf_{T\to\infty}
\left[
\frac{\tau_{1}-\tau_{0}}{T}\mathbbold{1}_{\{i(T)\geq 1\}}
+\frac{1}{1+\frac{\sum_{i=1}^{i(T)} (\tau_{2i}-\tau_{2i-1})}{\sum_{i=1}^{i(T)} (\tau_{2i+1}-\tau_{2i})}}
\right]\\
&=&\liminf_{T\to\infty}\frac{1}{1+\frac{\sum_{i=1}^{i(T)} (\tau_{2i}-\tau_{2i-1})}{\sum_{i=1}^{i(T)} (\tau_{2i+1}-\tau_{2i})}}
\end{eqnarray*}
{By taking the expectation of both sides and further combining Eq. \mbox{(\ref{eq lower bound})}, we will get the result of (3).}
\end{proof}

\begin{thm}
{For a stochastic equation in the form of \mbox{(\ref{StochasticEquation})},
if there exists a nonnegative function $V(\bm{x})\in
\mathscr{C}^{2}(\mathbb{R}^{n};\mathbb{R}_{\geq 0})$ satisfying the
following conditions:}
\begin{itemize}
  \item {$\lim_{\|\bm{x}\|_{2}\to\infty} V(\bm{x})=\infty$;}
  \item {$\exists~\bm{x}_{\tilde{R}}\in\mathbb{R}^n,~ \tilde{R}\textgreater 0~and~k\textgreater 0,~if~\|\bm{x}-\bm{x}_{\tilde{R}}\|_{2}\geq \tilde{R},~then~\mathcal{L}
  [V(\bm{x})]<-k$;}
  \item {$\exists~\epsilon\textgreater 0$, if $\|\bm{x}-\bm{x}_{\tilde{R}}\|_2< \tilde{R}+\epsilon$, then $rank(\bm{h}(\bm{x})\bm{h}^{\top}(\bm{x}))=n$;}
\end{itemize}
{then for any Borel subset $\mathbb{B}$ satisfying $V_{\mathbb{B}}>V_{0}$, we have}
\begin{equation}{\label{eq estimation of pi}}
\pi\left(\mathbb{B}\right) \geq
\frac{k\left(V_{\mathbb{B}}-V_{0}\right)}{2 C V_{\mathbb{B}}+k\left(V_{\mathbb{B}}-V_{0}\right)}.
\end{equation}
{where $C=\sup_{\bm{x}\in\mathbb{R}^{n}}\mathcal{L}[V(\bm{x})]$,
$V_{\mathbb{B}}=\inf_{\bm{x}\notin\mathbb{B}} V(x)$
and $V_{0}=\sup_{\|\bm{x}-\bm{x}_{\tilde{R}}\|_{2}\leq \tilde{R}} V(x)$.}
\end{thm}
\begin{proof}
{If $\pi(\mathbb{B})=1$, then the result holds true automatically. So we only prove the result in the case of $\pi(\mathbb{B})<1$ which implies $\pi\left(\{V(\bm{x})>V_{\mathbb{B}}\}\right)>0$.}

{From $\lim_{\|\bm{x}\|_{2}\to\infty} V(\bm{x})=\infty$ and $V_{0}=\sup_{\|\bm{x}-\bm{x}_{\tilde{R}}\|_{2}\leq \tilde{R}} V(x)$, we obtain} $$\{\bm{x}|V(\bm{x})\geq V_0\}\subseteq \{\bm{x}|~\|\bm{x}-\bm{x}_{\tilde{R}}\|_{2}\textgreater \tilde{R}\} $$
{Therefore, when $V(\bm{x})\geq V_0$, $\mathcal{L}[V(\bm{x})]<-k$.
    Let $V_{1}$ and $V_{2}$ mentioned in \textit{Lemma 3} be chosen as $V_{0}$ and $V_{\mathbb{B}}$.
    Then, by ergodicity in \textit{Theorem} \mbox{\ref{WeakStability}}, Eq. (\mbox{\ref{eq infinite loops}}) is achieved.
    }

{According to \textit{Theorem} \mbox{\ref{WeakStability}} and Eq. \mbox{(\ref{Eq global ergodicity})} we have}
\begin{eqnarray*}
\text{E}\left[\liminf_{T\to\infty}
\frac{1}{T}\int_{0}^{T}\mathbbold{1}_{\left\{V\left(\bm{x}(t)\right)\geq V_{\mathbbold{B}}\right\}} dt\right]
&=&\liminf_{T\to\infty}
\frac{1}{T}\int_{0}^{T}\mathbbold{1}_{\left\{V\left(\bm{x}(t)\right)\geq V_{\mathbbold{B}}\right\}} dt\\
&=&\pi\left(V\left(\bm{x}\right)\geq V_{\mathbbold{B}}\right)
\end{eqnarray*}
{By applying the result (3) of \textit{Lemma 3} we can get}
\begin{equation*}
\pi\left(V\left(\bm{x}\right)\geq V_{\mathbbold{B}}\right)
\leq
\frac{2 C V_{\mathbb{B}}}{2 C V_{\mathbb{B}}+k\left(V_{\mathbb{B}}-V_{0}\right)}
\end{equation*}
{Note that $\{\bm{x}|V(\bm{x})<V_\mathbbold{B}\}\subseteq \mathbbold{B}$, we thus have}
\begin{equation*}
\pi(\mathbbold{B})
\geq \pi\left(V(\bm{x})<V_\mathbbold{B}\right)
= 1 - \pi\left(V\left(\bm{x}\right)\geq V_{\mathbbold{B}}\right)
\geq
\frac{k\left(V_{\mathbb{B}}-V_{0}\right)}{2 C V_{\mathbb{B}}+k\left(V_{\mathbb{B}}-V_{0}\right)}
\end{equation*}
\end{proof}

\begin{rem}{\label{rem shape the invariant measure}}
{The above theorem reveals that the invariant measure can be shaped by controlling the change rates $k$ and $C$ of $V(\bm{x})$ together with the ball radius $\tilde{R}$ that takes effect through affecting $V_0$ in terms of $V_{0}=\sup_{\|\bm{x}-\bm{x}_{\tilde{R}}\|_{2}\leq \tilde{R}} V(x)$. When $\tilde{R}$ is fixed, the invariant measure will become larger if $k$ increases and$\backslash$or $C$ decreases. The larger the invariant measure is, the greater possibility the trajectory of the state fastens on the region near the desired state, i.e., the more stable the stochastic system is in the weak sense.}
\end{rem}

\subsection{Stochastic weak passivity theorems}
{Now, we are able to tackle the problem of
stabilizing the stochastic systems in weak sense based on the
stochastic weak passivity. We name the main result as stochastic
weak passivity theorems in the context. Here, concern is only given
to the implicit negative proportional controller
$\bm{u}=-\bm{K}\bm{y}(\bm{x},\bm{u})$ for the purpose of
stabilization, where $\bm{K}$ is a positive definite matrix with
suitable dimension. }

\begin{thm}[Stochastic Weak Passivity Theorem $1$]{\label{StochasticWeakPassivityTheorem}}
For a stochastic differential system $\Sigma_S$ in the form of Eq.
(\ref{StochasticSystem}), assume that there exists a {radially unbounded} storage
function
$V(\bm{x})\in\mathscr{C}^{2}(\mathbb{R}^{n};\mathbb{R}_{\geq 0})$
suggesting it to be stochastically weakly passive with the
stochastic passive radius $R$ and the desired state $\bm{x}^\dag$.
Also, we suppose that there exists a negative proportional
controller $\bm{u}(\bm{x})=-\bm{K}\bm{y}(\bm{x},\bm{u})$ connected
with $\Sigma_S$ in feedback so that
\begin{itemize}
  \item $\|\bm{y}\big(\bm{x},\bm{u}(\bm{x})\big)\|_{2}$ be bounded away from zero when
  $\|\bm{x}-\bm{x}^\dag\|_{2}\textgreater R$;
  \item $\exists~\epsilon\textgreater 0$, rank($\bm{h}\big(\bm{x},\bm{u}(\bm{x})\big)\bm{h}^{\top}\big(\bm{x},\bm{u}(\bm{x})\big)=n$ when $\|\bm{x}-\bm{x}^\dag\|_2 < R+\epsilon$.
\end{itemize}
Then there exists an unique finite invariant measure $\pi$, and
moreover, for any Borel subset $\mathbb{A}$ with zero $\pi$-measure
boundary
$$\lim_{t\to\infty}\mathcal{P}(t,\bm{x}(0),\mathbb{A})=\pi(\mathbb{A}),~~~~\forall~
\bm{x}(0)\in\mathbb{R}^{n}$$
{and for any Borel subset $\mathbbold{B}$}
\begin{equation*}
\lim_{T\to\infty} \frac{1}{T}\int_{0}^{T} \mathbbold{1}_{\left\{\bm{x}(t)\in \mathbb{B}\right\}} dt =\pi(\mathbb{B}),~~~\mathrm{a.s.}~~\forall~
\bm{x}(0)\in\mathbb{R}^{n}
\end{equation*}
That is to say $\Sigma_S$ being
globally asymptotical stable in weak sense.
\end{thm}

\begin{proof}
From the given conditions {and the definition of $\mathcal{L}[\cdot]$ in Eq. \mbox{(\ref{InfinitesimalGenerator})}}, for any state outside the ball
$\|\bm{x}-\bm{x}^\dag\|_{2}\geq R$ we have
\begin{equation*}
\mathcal{L}[V(\bm{x})]\leq \bm{y}^{\top}\bm{u} =
-\bm{y}^{\top}(\bm{x},\bm{u})\bm{K}\bm{y}(\bm{x},\bm{u})\leq-\lambda_{\min}\|\bm{y}\big(\bm{x},\bm{u}(\bm{x})\big)\|_{2}^{2}
\end{equation*}
where $\lambda_{\min}\textgreater 0$ is the minimum eigenvalue of
$\bm{K}$. Since $\|\bm{y}\big(\bm{x},\bm{u}(\bm{x})\big)\|_{2}$ is
bounded away from zero when $\|\bm{x}-\bm{x}^\dag\|_{2}\textgreater
R$, we could find a positive number $k$ such that for any
$\|\bm{x}-\bm{x}^\dag\|_{2}\textgreater R$
\begin{equation*}
\mathcal{L} [V(\bm{x})]\textless -k
\end{equation*}
Combining it with the non-singularity of
$\bm{h}\big(\bm{x},\bm{u}(\bm{x})\big)\bm{h}^{\top}\big(\bm{x},\bm{u}(\bm{x})\big)$
when $\|\bm{x}-\bm{x}^\dag\|_2 < R+\epsilon$ yields the results
(based on \textit{Theorem} $5$).
\end{proof}

Note that the above theorem works when the output norm is bounded
way from zero outside a ball. In fact, this condition is not
necessary. It may be replaced by setting the system to be strictly
state stochastic weak passive.

\begin{thm}[Stochastic Weak Passivity Theorem $2$]{\label{StochasticWeakPassivityTheorem2}}
Consider a strictly state stochastic weak passive system under the
conditions of radially unbounded $V(\bm{x})$, $R$, $\bm{x}^\dag$ and $\delta$. Suppose
that there exists a negative proportional controller
$\bm{u}(\bm{x})=-\bm{K}\bm{y}(\bm{x},\bm{u})$ connecting with the
system in feedback so that
\begin{itemize}
  \item $\exists~\epsilon\textgreater 0$, rank($\bm{h}\big(\bm{x},\bm{u}(\bm{x})\big)\bm{h}^{\top}\big(\bm{x},\bm{u}(\bm{x})\big)=n$ when $\|\bm{x}-\bm{x}^\dag\|_2 < R+\epsilon$.
\end{itemize}
Then there exists an unique finite invariant measure $\pi$, and
moreover, for any Borel subset $\mathbb{A}$ with zero $\pi$-measure
boundary
$$\lim_{t\to\infty}\mathcal{P}(t,\bm{x}(0),\mathbb{A})=\pi(\mathbb{A}),~~~~\forall~
\bm{x}(0)\in\mathbb{R}^{n}$$
{and for any Borel subset $\mathbbold{B}$}
\begin{equation*}
\lim_{T\to\infty} \frac{1}{T}\int_{0}^{T} \mathbbold{1}_{\left\{\bm{x}(t)\in \mathbb{B}\right\}} dt =\pi(\mathbb{B}), ~~~\mathrm{a.s.}~~\forall~
\bm{x}(0)\in\mathbb{R}^{n}
\end{equation*}
That is to say $\Sigma_S$ being
globally asymptotical stable in weak sense.
\end{thm}

\begin{proof}
From the known conditions {and the definition of $\mathcal{L}[\cdot]$ in Eq. \mbox{(\ref{InfinitesimalGenerator})}}, we have at
$\|\bm{x}-\bm{x}^\dag\|_{2}\geq R$
\begin{equation*}
\mathcal{L}[V(\bm{x})]\leq
\bm{y}^{\top}\bm{u}-\delta\|\bm{x}-\bm{x}^{\dag}\|_2^{2} =
-\bm{y}^{\top}(\bm{x},\bm{u})\bm{K}\bm{y}(\bm{x},\bm{u})-\delta\|\bm{x}-\bm{x}^{\dag}\|_2^{2}
\leq-\delta R^{2}
\end{equation*}
i.e., $\exists~ k\textgreater 0,~\mathcal{L} [V(\bm{x})]\textless
-k$. Further applications of \textit{Theorem} $5$ yield the results
immediately.
\end{proof}

\begin{rem}
{Clearly, the change rate $k$ of the energy function $V(\bm{x})$ is closely dependent on the the feedback gain matrix $\bm{K}$. When $\bm{K}$ is designed to be stronger, $k$ will be larger. This can lead to the increase of the invariant measure $\pi\left(\mathbb{B}\right)$, and further the state evolution taking place within a more intensive region around the desired state $\bm{x}^\dag$. Therefore, it is an effective way for shaping the invariant measure $\pi\left(\mathbb{B}\right)$ purposefully to strengthen the feedback gain matrix $\bm{K}$.}
\end{rem}

\begin{rem}
Stochastic weak passivity theorems suggest
sufficient conditions to stabilize a stochastic differential system
in weak sense. Although some items are difficult to realize, such as
the non-singularity of
$\bm{h}\big(\bm{x},\bm{u}(\bm{x})\big)\bm{h}^{\top}\big(\bm{x},\bm{u}(\bm{x})\big)$
when $\|\bm{x}-\bm{x}^\dag\|_2 < R+\epsilon$, some ones are
relatively weak, e.g., the energy function $V(\bm{x})$ only requires
to be positive semi-definite, and the simple proportional controller
is qualified.
\end{rem}

{To weaken the harsh condition on the non-singularity
of
$\bm{h}\big(\bm{x},\bm{u}(\bm{x})\big)\bm{h}^{\top}\big(\bm{x},\bm{u}(\bm{x})\big)$
in the Stochastic weak passivity theorem, we might separate those
linear independent rows from $\bm{h}\big(\bm{x},\bm{u}(\bm{x})\big)$
to construct a new diffusion term
$\bm{h}_1(\cdot,\cdot)\in\mathbb{R}^{n_1\times r}$. Clearly, it is
much easier to realize
$\bm{h}_1(\cdot,\cdot)\bm{h}_1^\top(\cdot,\cdot)$ to be full-rank
than
$\bm{h}\big(\bm{x},\bm{u}(\bm{x})\big)\bm{h}^{\top}\big(\bm{x},\bm{u}(\bm{x})\big)$
to be nonsingular. For this purpose, we first define a
transformation.}
\begin{definition}[Decomposition Transformation]
A homeomorphism
$\boldsymbol{\Phi}(\bm{x})\in\mathscr{C}^{2}(\mathbb{R}^{n};\mathbb{R}^n)$,
expressed as
\begin{equation}\label{DecompositionTransformation}
\boldsymbol{\Phi}(\bm{x})=\left(
\begin{array}{c}
\bar{\bm{x}}_{1}\\
\bar{\bm{x}}_{2}
\end{array}
\right)
\end{equation}
is called a decomposition transformation of system $\Sigma_{S}$ if it can transform the
stochastic differential system $\Sigma_{S}$ equipped with Eq.
(\ref{StochasticSystem}) into two subsystems: one is a stochastic
differential system $\Sigma_{subS}$, the other is a deterministic
system $\Sigma_{subD}$. Here,
$\bar{\bm{x}}_{1}\in\mathbb{R}^{n_1},~\bar{\bm{x}}_{2}\in\mathbb{R}^{n_2}$
and $n_1+n_2=n$. These two subsystems are written respectively as
\begin{eqnarray}
 \Sigma_{subS}:&~&\left\{
\begin{array}{lll}\label{subSystem1}
\dd \bar{\bm{x}}_{1}&=&\bm{f}_{1}(\bar{\bm{x}}_{1},\bm{u}) \dd t+ \bm{h}_{1}(\bar{\bm{x}}_{1},\bm{u}) \dd \boldsymbol{\omega} \\
\bm{y}_1&=&\bm{s}_{1}(\bar{\bm{x}}_{1},\bm{u})
\end{array}
\right.\\
\Sigma_{subD}:&~&\left\{
\begin{array}{lll}\label{subSystem2}
\dd\bar{\bm{x}}_{2}&=& \mathbbold{0}_{n_2}\dd t \\
\bm{y}_{2}&=&\bm{s}_{2}(\bar{\bm{x}}_{2},\bm{u})
\end{array}
\right..
\end{eqnarray}
where the drift term
$\bm{f}_{1}(\bar{\bm{x}}_{1},\bm{u})\in\mathbb{R}^{n_1}$ and
diffusion term
$\bm{h}_{1}(\bar{\bm{x}}_{1},\bm{u})\in\mathbb{R}^{n_1\times r}$ in
$\Sigma_{subS}$ are both locally Lipschitz continuous.
\end{definition}

\begin{rem}
{The diffusion term
$\bm{h}_{1}(\bar{\bm{x}}_{1},\bm{u})$ in $\Sigma_{subS}$ can be
extracted from $\bm{h}\big(\bm{x},\bm{u}(\bm{x})\big)$ to the
greatest extent according to the rank so that the rank of
$\bm{h}_{1}(\bar{\bm{x}}_{1},\bm{u})$ may reach $n_1$.
$\Sigma_{subD}$ is obviously a fixed point, and is certainly stable.
Thus, the stabilization of $\Sigma_{S}$ may be realized by
stabilizing $\Sigma_{subS}$.}
\end{rem}

\begin{rem}{\label{rem explaination of how to find decompation transformation}}
{In practice, the defined decomposition transformation \mbox{(\ref{DecompositionTransformation})} is not difficult to be constructed. For many nonlinear stochastic systems, the state evolution really takes place in an invariant manifold, denoted by $\mathcal{M}$, but not in the whole state space $\mathbb{R}^n$. For example, the state of a chemical reaction network will evolve in a positive stoichiometric compatibility class, which is a subset of $\mathbb{R}^n$. Therefore, an immediate idea to construct the decomposition transformation is to decompose the state space $\mathbb{R}^n$ into the invariant manifold $\mathcal{M}$ and its orthogonal complement $\mathcal{M}^\perp$. The projection from $\bm{x}\in\mathbb{R}^n$ to $\bar{\bm{x}}_1\in\mathcal{M}$ leads to the stochastic subsystem $\Sigma_{subS}$ while the projection from $\bm{x}\in\mathbb{R}^n$ to $\bar{\bm{x}}_2\in\mathcal{M}^\perp$ induces the fixed point subsystem $\Sigma_{subD}$. }
\end{rem}

Utilizing the decomposition transformation, we can give the refined
stochastic weak passivity theorem.
\begin{lemma}\label{TransformeStateManifold}
For a stochastic differential system (\ref{StochasticSystem}) if
there exists a decomposition transformation
$\boldsymbol{\Phi}(\bm{x})$ transform the system into two subsystems
$\Sigma_{subS}$ (\ref{subSystem1}) and $\Sigma_{subD}$
(\ref{subSystem2}), then the state $\bm{x}(t)$ staring from any
initial state $\bm{x}(0)$ satisfies
\begin{equation*}
\bm{x}(t) \in
\boldsymbol{\Phi}^{-1}\big(\mathbb{R}^{n_1}\otimes\{\bar{x}_{2}(0)\}\big),
~~~\forall~ t\in[0,\infty)
\end{equation*}
\end{lemma}
\begin{proof}
Since
$\boldsymbol{\Phi}(\bm{x}(0))=\bar{\bm{x}}(0)=\bar{\bm{x}}_{1}(0)\otimes\bar{\bm{x}}_{2}(0)$,
also since $\bar{\bm{x}}_2(t)=\bar{\bm{x}}_2(0)$, the result is
true.
\end{proof}

\begin{rem}
It is clear that $\bar{\bm{x}}_2(0)$ can be any element in
$\mathbb{R}^{n_2}$. This in turn means that for any
$\bm{z}\in\mathbb{R}^{n_2}$ there is a manifold
$\boldsymbol{\Phi}^{-1}\big(\mathbb{R}^{n_1}\otimes\{\bm{z}\}\big)$
defined. For simplicity of notation, we identify this manifold by
$\mathbb{R}^{n}_{\bm{z}}$.
\end{rem}

\begin{thm}[Refined Stochastic Weak Passivity Theorem]{\label{RefinedSWPT}}
Assume that the system $\Sigma_{S}$ can be decomposed into two
subsystems $\Sigma_{subS}$ (\ref{subSystem1}) and $\Sigma_{subD}$
(\ref{subSystem2}) using the decomposition transformation
$\boldsymbol{\Phi}(\bm{x})$. Further, suppose that there exist a
{radially unbounded} storage function
$V(\bar{\bm{x}}_{1})\in\mathscr{C}^{2}(\mathbb{R}^{n_1};\mathbb{R}_{\geq
0})$ suggesting $\Sigma_{subS}$ to be stochastic weak passivity with
the stochastic passive radius $R$ and the desired state
$\bar{\bm{x}}^\dag_{1}$, and a negative implicit proportional
controller
$\bm{u}(\bar{\bm{x}}_{1})=-\bm{K}\bm{y}_{1}(\bar{\bm{x}}_{1},\bm{u})$
so that
\begin{itemize}
  \item $\|\bm{y}_{1}(\bar{\bm{x}}_{1},\bm{u})\|_{2}$ is bounded away from $0$ when
  $\|\bar{\bm{x}}_{1}-\bar{\bm{x}}^\dag_{1}\|>R$; Or the system is strictly state stochastic weak passive.
  \item $\exists~\epsilon\textgreater 0,~rank(\bm{h}_{1}(\bar{\bm{x}}_{1},\bm{u})\bm{h}_{1}^{\top}(\bar{\bm{x}}_{1},\bm{u}))=n_1$ when $\|\bar{\bm{x}}_{1}-\bar{\bm{x}}^\dag_{1}\|_2\textless R+\epsilon$.
\end{itemize}
Then there exist an invariant measure $\pi$ and a corresponding
manifold $\mathbb{R}^{n}_{\pi}$ such that for any Borel subset
$\mathbb{A}$ with zero $\pi$-measure boundary
\begin{equation*}
\lim_{t\to\infty}\mathcal{P}\big(t,\bm{x}(0),\mathbb{A}\big)=\pi(\mathbb{A}),~\forall~\bm{x}(0)\in\mathbb{R}^{n}_{\pi}
\end{equation*}
{and for any Borel subset $\mathbbold{B}$}
\begin{equation*}
\lim_{T\to\infty} \frac{1}{T}\int_{0}^{T} \mathbbold{1}_{\left\{\bm{x}(t)\in \mathbb{B}\right\}} dt =\pi(\mathbb{B}),~~~\mathrm{a.s.}~~\forall~
\bm{x}(0)\in\mathbb{R}^{n}_{\pi}
\end{equation*}
when the controller $\bm{u}(\bar{\bm{x}}_{1})$ is connected with
$\Sigma_{subS}$ in feedback. That is to say $\Sigma_{S}$ being
locally asymptotically weakly stable.
\end{thm}

\begin{proof}
{First, we consider the existence of invariant
measure under the transformed coordinate $\bar{\bm{x}}$.}

{For $\Sigma_{subS}$ connected by the controller
$\bm{u}(\bar{\bm{x}}_{1})$ in feedback, by stochastic weak passivity theorems (\textit{Theorem}
\ref{StochasticWeakPassivityTheorem} and \textit{Theorem} \ref{StochasticWeakPassivityTheorem2}) there exists an unique finite
invariant measure $\pi_{1}$ so that for any Borel subset
$\mathbb{A}_{1}\subset\mathbb{R}^{n_{1}}$ with boundary
$\Gamma_{1}$, when $\pi_{1}(\Gamma_{1})=0$ we have
\begin{equation}{\label{eq convergence of pi1}}
\lim_{t\to\infty}
\mathcal{P}(t,\bar{\bm{x}}_{1}(0),\mathbb{A}_{1})=\pi_{1}(\mathbb{A}_{1}),~\forall~
\bar{\bm{x}}_{1}(0)\in\mathbb{R}^{n_{1}}
\end{equation}}
{and for any Borel subset $\mathbbold{B}_{1}\subset\mathbb{R}^{n_{1}}$}
\begin{equation*}
\lim_{T\to\infty} \frac{1}{T}\int_{0}^{T} \mathbbold{1}_{\left\{\bar{\bm{x}}_{1}(t)\in \mathbb{B}_{1}\right\}} \dd t =\pi_{1}(\mathbb{B}_{1}), ~~~\mathrm{a.s.}~~\forall~
\bar{\bm{x}}_{1}(0)\in\mathbb{R}^{n_{1}}
\end{equation*}

For $\Sigma_{subD}$, define $\pi_{2}$ be a measure on
\big($\mathbb{R}^{n_{2}}$,$\mathcal{B}(\mathbb{R}^{n_{2}})$\big)
that satisfies
\begin{equation*} \pi_{2}(\mathbb{A}_2)= \left\{
\begin{array}{ll}
1 &{\bm{z}} \in\mathbb{A}_{2}\\
0 &{\bm{z}} \notin \mathbb{A}_{2}
\end{array}
\right.
\end{equation*}
where ${\bm{z}}$ is a fixed point in $\mathbb{R}^{n_{2}}$ and $\mathbb{A}_2$ is a Borel subset of $\mathbb{R}^{n_{2}}$, then for all $t$ and the initial condition $\bm{z}$ we
have
\begin{equation*}
\mathcal{P}\big(t,\bar{\bm{x}}_{2}(0),\mathbb{A}_{2}\big)=\text{P}\big(\bar{\bm{x}}_{2}(t)\in
\mathbb{A}_2 | \bar{\bm{x}}_{2}(0)=\bm{z}\big)=\pi_{2}(\mathbb{A}_{2})
\end{equation*}

Consider the product measure of $\pi_{1}$ and $\pi_{2}$, denoted by
$\pi_{3}$, i.e.,
$$\pi_{3}(\star\otimes
*)=\pi_{1}(\star)\times\pi_{2}(*)$$
Note that the existence and uniqueness of $\pi_{3}$ are guaranteed
by Hahn-Kolmogorov theorem and $\sigma$-finite property,
respectively. Imposing $\pi_{3}$ on the set
$\mathbb{R}^{n_1}\otimes\big\{\bm{z}\big\}^{c}$ yields
\begin{equation}{\label{Pi3Zeromeasure}}
\pi_{3}\Big(\mathbb{R}^{n_1}\otimes\big\{\bm{z}\big\}^{c}\Big)=\pi_{1}\Big(\mathbb{R}^{n_1}\Big)\times\pi_{2}\Big(\big\{\bm{z}\big\}^{c}\Big)=0.
\end{equation}
where $\big\{\bm{z}\big\}^{c}$ is the complementary set
of $\big\{\bm{z}\big\}$ in $\mathbb{R}^{n_2}$.

For any Borel subset $\mathbb{A}\subset\mathbb{R}^{n}$, we could
express it as
$$\mathbb{A}=\Bigg[\mathbb{A}\cap\Big(\mathbb{R}^{n_{1}}\otimes\big\{\bm{z}\big\}\Big)\Bigg]
    \cup
    \Bigg[\mathbb{A}\cap\Big(\mathbb{R}^{n_{1}}\otimes\big\{\bm{z}\big\}\Big)^{c}\Bigg]$$
which may be further rewritten by defining a map
$\boldsymbol{\Psi}(\bar{\bm{x}}_{1}\otimes\bar{\bm{x}}_{2})\triangleq\bar{\bm{x}}_{1}$
as
\begin{eqnarray}{\label{eq divide A}}
\mathbb{A}&=&
\Bigg[\boldsymbol{\Psi}(\mathbb{A})\otimes\big\{\bm{z}\big\}\Bigg]
     \cup
     \Bigg[\mathbb{A}\cap\Big(\mathbb{R}^{n_{1}}\otimes\big\{\bm{z}\big\}\Big)^{c}\Bigg]
\end{eqnarray}
Here,
$\boldsymbol{\Psi}(\mathbb{A})=\left\{\boldsymbol{\Psi}(\bar{\bm{x}}_{1}\otimes\bar{\bm{x}}_{2})
\big| \bar{\bm{x}}_{1}\otimes\bar{\bm{x}}_{2} \in
\mathbb{A}\cap\left(\mathbb{R}^{n_{1}}\otimes\big\{\bm{z}\}\right)\right\}$.
Then we have
\begin{eqnarray}{\label{PI3TransitionMeasure}}
&&\mathcal{P}(t,\bar{\bm{x}},\mathbb{A})\notag\\
&=&\mathcal{P}\Big(t,\bar{\bm{x}},\big[\boldsymbol{\Psi}(\mathbb{A})\otimes\big\{\bm{z}\big\}\big]
     \cup
     \big[\mathbb{A}\cap\big(\mathbb{R}^{n_{1}}\otimes\big\{\bm{z}\big\}\big)^{c}\big]\Big)\notag\\
&=&\mathcal{P}\Big(t,\bar{\bm{x}},\boldsymbol{\Psi}(\mathbb{A})\otimes\big\{\bm{z}\big\}\Big)+\mathcal{P}\Big(t,\bar{\bm{x}},\mathbb{A}\cap\big(\mathbb{R}^{n_{1}}\otimes\big\{\bm{z}\big\}\big)^{c}\Big)
\end{eqnarray}
Note that for all $\bar{\bm{x}}(0)\in\mathbb{R}^{n_1}\otimes\{\bm{z}\}$ the second term in the above equality satisfies
\begin{eqnarray}{\label{ineq transition measure equals 0}}
&&\mathcal{P}\Bigg(t,\bar{\bm{x}}(0),\mathbb{A}\cap\Big(\mathbb{R}^{n_{1}}\otimes\big\{\bm{z}\big\}\Big)^{c}\Bigg)\notag\\
&\leq&\mathcal{P}\Big(t,\bar{\bm{x}}(0),\mathbb{R}^{n_{1}}\otimes\big\{\bm{z}\big\}^{c}\Big)\notag\\
&=&\mathcal{P}\Big(\bar{\bm{x}}_{2}(t)\notin\big\{\bm{z}\big\}\Big)\notag\\
&=&0
\end{eqnarray}
then we have
\begin{eqnarray*}
&&\int_{\bar{\bm{x}}\in\mathbb{R}^{n}}\mathcal{P}(t,\bar{\bm{x}},\mathbb{A})\pi_{3}(d\bar{\bm{x}})\\
&\xlongequal[Lemma~
\ref{TransformeStateManifold}]{\text{Eqs.~(\ref{PI3TransitionMeasure})~\&~(\ref{ineq
transition measure equals 0})}}&
\int_{\bar{\bm{x}}\in\mathbb{R}^{n_{1}}\otimes\{\bm{z}\}}\mathcal{P}\Big(t,\bar{\bm{x}},\boldsymbol{\Psi}(\mathbb{A})\otimes\big\{\bm{z}\big\}\Big)\pi_{3}(d\bar{\bm{x}})\\
&\xlongequal{\text{~\quad \quad \quad \quad\quad\quad \quad
\quad~}}&
\int_{\bar{\bm{x}}_{1}\in\mathbb{R}^{n_{1}}}\mathcal{P}\Big(t,\bar{\bm{x}}_{1},\boldsymbol{\Psi}(\mathbb{A})\Big)\pi_{1}(d\bar{\bm{x}}_{1})\\
&\xlongequal{~~~~~\text{Eq. (\ref{Def-InvariantMeasure})}~~~~~~}& \pi_{1}\big(\boldsymbol{\Psi}(\mathbb{A})\big)\\
&\xlongequal{\text{Eqs.~(\ref{eq divide A})~\&~(\ref{ineq transition
measure equals 0})}}& \pi_{3}(\mathbb{A})
\end{eqnarray*}
Hence, $\pi_{3}$ is invariant under the coordinate $\bar{\bm{x}}$.

{Next, we discuss the convergency of $\pi_{3}$ under
the transformed coordinate $\bar{\bm{x}}$.}

Let the boundary of
$\boldsymbol{\Psi}(\mathbb{A})\subset\mathbb{R}^{n_1}$ be
$\Gamma_{1}$, then for $\forall~ \bar{\alpha}\in \Gamma_1$ there
exist two sequences of points $\{\bar{\beta}_{i}\}_{i=1}^\infty$
\big($\bar{\beta}_{i}\notin\boldsymbol{\Psi}(\mathbb{A})$\big) and
$\{\bar{\gamma}_{i}\}_{i=1}^\infty$
\big($\bar{\gamma}_{i}\in\boldsymbol{\Psi}(\mathbb{A})$\big) such
that
$$\lim_{i\to\infty}\bar{\beta}_{i}=\lim_{i\to\infty}\bar{\gamma}_{i}=\bar{\alpha}$$
Hence, for any point $\bar{\alpha}\otimes\bm{z}$ in
$\Gamma_{1}\otimes\big\{\bm{z}\big\}$, there exist two
sequences of points
$\big\{\bar{\beta}_i\otimes\bm{z}\}_{i=1}^\infty
\big(\bar{\beta}_{i}\otimes\bm{z}\notin
\mathbb{A}\big)$ and
$\big\{\bar{\gamma}_{i}\otimes\bm{z}\big\}_{i=1}^\infty
\big(\bar{\gamma}_{i}\otimes\bm{z}\in \mathbb{A}\big)$
such that
$$\lim_{i\to\infty}\bar{\beta}_{i}\otimes\bm{z}=\lim_{i\to\infty}\bar{\gamma}_{i}\otimes\bm{z}=\bar{\alpha}\otimes\bm{z}$$
Further let $\Gamma$ denote the boundary of $\mathbb{A}$, then we
have $\Gamma_{1}\otimes\{\bm{z}\}\subset\Gamma$ and
$$\pi_{3}(\Gamma)=\pi_{1}(\Gamma_{1})$$

Assume $\pi_{3}(\Gamma)=0$, i.e., $\pi_{1}(\Gamma_{1})=0$, with
which we get for
$\forall~\bar{\bm{x}}(0)=\bar{\bm{x}}_{1}(0)\otimes\bar{\bm{x}}_{2}(0)\in
\mathbb{R}^{n_1}\otimes\{\bm{z}\}$
\begin{eqnarray*}
&&\lim_{t\to\infty}\mathcal{P}\big(t,\bar{\bm{x}}(0),\mathbb{A}\big)\\
&\xlongequal{\text{Eqs.~(\ref{eq divide A})~\&~(\ref{ineq transition measure equals 0})}}& \lim_{t\to\infty}\mathcal{P}\Big(t,\bar{\bm{x}}(0),\boldsymbol{\Psi}(\mathbb{A})\otimes\big\{\bm{z}\}\Big)\\
&\xlongequal{~~~~~~\quad \quad \quad \quad\quad}&
\lim_{t\to\infty}\mathcal{P}\big(t,\bar{\bm{x}}_{1}(0),\boldsymbol{\Psi}(\mathbb{A})\big)\\
&\xlongequal{~~~\quad\text{Eq.~(\ref{eq convergence of pi1})}\quad~~~}& \pi_{1}\big(\boldsymbol{\Psi}(\mathbb{A})\big)\\
&\xlongequal{\text{Eqs.~(\ref{eq divide A})~\&~(\ref{ineq transition
measure equals 0})}}& \pi_{3}(\mathbb{A})
\end{eqnarray*}
This complete the proof of the convergence of $\pi_3$.

{Finally, we consider the existence and convergence
of invariant measure under original coordinate.}

Let a measure $\pi$ satisfy
$\pi(\mathbb{A})=\pi_{3}(\boldsymbol{\Phi}(\mathbb{A}))$ for all
$\mathbb{A}\in\mathcal{B}(\mathbb{R}^{n})$. We have
\begin{eqnarray*}
\int_{\bm{x}\in\mathbb{R}^{n}}\mathcal{P}(t,\bm{x},\mathbb{A})\pi(d\bm{x})
=\int_{\bar{\bm{x}}\in\mathbb{R}^{n}}\mathcal{P}\big(t,\bar{\bm{x}},\boldsymbol{\Phi}(\mathbb{A})\big)\pi_{3}(d\bar{\bm{x}})
=\pi_{3}\big(\boldsymbol{\Phi}(\mathbb{A})\big) =\pi(\mathbb{A})
\end{eqnarray*}
which means $\pi$ is invariant. Further let $\Gamma_{0}$ denote the
boundary of $\boldsymbol{\Phi}(\mathbb{A})$, then
$\boldsymbol{\Phi}^{-1}(\Gamma_{0})\subset\Gamma$ due to the
bicontinuity of $\boldsymbol\Phi$. Thus, if we assume $\pi(\Gamma)=0$,
then
$$\pi_3(\Gamma_{0})=\pi\big(\boldsymbol{\Phi}^{-1}(\Gamma_{0})\big)=0$$
Hence, for $\forall~\bm{x}(0)\in
\boldsymbol{\Phi}^{-1}\big(\mathbb{R}^{n_1}\otimes\{\bm{z}\}\big)=\mathbb{R}^{n}_{\bm{z}}(\text{denoted as}\mathbb{R}^{n}_{\pi})$
\begin{eqnarray*}
\lim_{t\to\infty}\mathcal{P}\big(t,\bm{x}(0),\mathbb{A}\big)&=&\lim_{t\to\infty}\mathcal{P}\big(t,\bar{\bm{x}}(0),\boldsymbol{\Phi}(\mathbb{A})\big)
=\pi_{3}\big(\boldsymbol{\Phi}(\mathbb{A})\big) =\pi(\mathbb{A})
\end{eqnarray*}
which shows the convergence of the transition measure.

{Similarly, we can prove the local ergodicity of the process.}
\end{proof}

\begin{rem}
{A point should be noted that the current invariant
measure is no longer unique. It is closely dependent on the initial
condition $\bm{x}(0)$. Hence, the refined stochastic weak passivity
theorem actually suggests the conditions of local asymptotic weak stability for
a stochastic differential system. }
\end{rem}

\section{Applications} In this section, the stochastic weak passivity
theory is applied to linear systems and a nonlinear process system.
\subsection{Application to linear systems}
Consider a representative linear time-invariant system described by
\begin{equation}\label{LinearSystem}
\left\{
\begin{array}{ccl}
\dd \bm{x} &=&(\bm{A}\bm{x}+\bm{B}\bm{u})\dd t + \boldsymbol{\sigma} \dd \boldsymbol{\omega}\\
\bm{y}     &=& \bm{C}\bm{x}
\end{array}
\right.
\end{equation}
where $\bm{A},~\bm{B},~\bm{C}$ and $\boldsymbol{\sigma}\neq 0$ are
constant matrices with suitable dimensions. For simplicity let
$\mathbbold{0}_{n}$ be the desired state.

Since the noise port $\boldsymbol{\sigma}\neq 0$, there does not
exist any Lyapunov function that could suggest this linear
system to be {globally} stable in the sense of probability. However, it is
possible for this system to reach stochastic asymptotic weak stability.

\begin{thm}
For a linear system described by Eq. (\ref{LinearSystem}), if there
exists a positive definite matrix $\bm{D}$ with suitable dimension
such that
\begin{itemize}
  \item $\bm{C}=\bm{B}^\top\bm{D}$;
  \item $\bm{DA}+\bm{A}^\top\bm{D}$ is negative definite.
\end{itemize}
then the system is stochastically weak passive.
\end{thm}

\begin{proof}
Let $V(\bm{x})=\frac{1}{2}\bm{x}^\top\bm{D}\bm{x}$, then
\begin{eqnarray*}
\mathcal{L}[V(\bm{x})]
&=&\bm{x}^\top\bm{D}(\bm{Ax}+\bm{Bu})+\frac{1}{2}\text{tr}\{\bm{D}\boldsymbol{\sigma}\boldsymbol{\sigma}^\top\} \\
&=&\bm{x}^{\top}\bm{D}\bm{Bu}+\bm{x}^{\top}\bm{DA}\bm{x}
+\frac{1}{2}\text{tr}\left\{\bm{D}\boldsymbol{\sigma}\boldsymbol{\sigma}^\top\right\}\\
&=&
\bm{y}^{\top}\bm{u}+\frac{1}{2}\bm{x}^{\top}(\bm{DA}+\bm{A}^\top\bm{D})\bm{x}
+\frac{1}{2}\text{tr}\left\{\bm{D}\boldsymbol{\sigma}\boldsymbol{\sigma}^\top\right\}
\end{eqnarray*}
Assume $\lambda_{\text{max}}$ to be the maximum eigenvalue of
$\bm{DA}+\bm{A}^\top\bm{D}$. Since the matrix is negative definite,
we have $\lambda_{\text{max}}\textless 0$ and
$\frac{1}{2}\bm{x}^{\top}(\bm{DA}+\bm{A}^\top\bm{D})\bm{x}\leq
\frac{1}{2}\lambda_{\text{max}}\|\bm{x}\|^{2}_2$. Note that
$\text{tr}\left\{\bm{D}\boldsymbol{\sigma}\boldsymbol{\sigma}^\top\right\}=\text{tr}\left\{\boldsymbol{\sigma}^\top\bm{D}\boldsymbol{\sigma}\right\}\geq
0$. Hence, as long as $\|\bm{x}\|_2\geq
\sqrt{\frac{\text{tr}\left\{\bm{D}\boldsymbol{\sigma}\boldsymbol{\sigma}^\top\right\}}{-\lambda_\text{max}}}$,
we get
\begin{eqnarray*}
\mathcal{L}[V(\bm{x})]&\leq&\bm{y}^{\top}\bm{u}
\end{eqnarray*}
which completes the proof.
\end{proof}

Clearly, the used Lyapunov function
$V(\bm{x})=\frac{1}{2}\bm{x}^\top\bm{D}\bm{x}$ is {radially unbounded}, so if the
noise port $\boldsymbol{\sigma}$ is full row rank and the
measurement matrix $\bm{C}$ is full column rank, then all conditions
required in \textit{Theorem} \ref{StochasticWeakPassivityTheorem}
are true. Namely, any negative proportional controller
$\bm{u}(\bm{x})=-\bm{K}\bm{y}(\bm{x},\bm{u})$ can globally stabilize
this class of stochastic linear systems in weak sense.

\subsection{Applications to a nonlinear process system}
Next, we manage to analyze the stochastic weak passivity of a
nonlinear process system and further stabilize it based on the
refined stochastic weak passivity theorem.

Consider a continuous stirred tank reactor (CSTR) in which a
first-order chemical reaction takes place
$$X_1\autorightarrow{}{}X_2$$
The pure component $X_1$ with a fixed concentration
$C_1^\text{in}$ is fed at the inlet of the reactor while the fluid
mixture of $X_1$ and $X_2$ is released from the outlet of the
reactor. For simplicity this process is assumed to be isothermal and
isometric. Furthermore, the volume flow rates of inflow and outflow
are regulated to be the same. In addition, the chemical reaction
goes on with a unstable reaction rate coefficient. Denote the
components concentrations of $X_1$ and $X_2$ by $x_1$ and $x_2$,
respectively, the volume flow rate of inflow or outflow by $q$, the
reaction rate coefficient by $k$ and the disturbance on $k$ by
$\sigma$ (note that $\sigma\textgreater 0$ and $\sigma\ll k$), then
we have the dynamical equation
\begin{equation}\label{CSTRModel}
\left\{
\begin{array}{lll}
\dd x_{1}&=& \left[-k{x}_{1}+\left(C_1^\text{in}-{x}_{1}\right)q\right]\dd t -\sigma{x}_{1}\dd \omega \\
\dd x_{2}&=& \left[k{x}_{1}-{x}_{2}q\right] \dd t +\sigma{x}_{1}\dd
\omega
\end{array}
\right.
\end{equation}

Construct the input-output pair as
\begin{equation*}
\left\{
\begin{array}{lll}
u&=&q-\frac{k{x}_{1}^{\dag}}{C_1^\text{in}-{x}_{1}^{\dag}} \\
y&=&\left({x}_{1}-{x}_{1}^{\dag}\right)\left(C_1^\text{in}-{x}_{1}\right)
\end{array}
\right.
\end{equation*}
where ${x}_{1}^{\dag}$ is the desired concentration of $X_1$
constrained by $0\textless {x}_{1}^{\dag} \textless C_1^\text{in}$.
Then the input-output representation of the stochastic CSTR model
(\ref{CSTRModel}) can be written as
\begin{equation*}
\left\{
\begin{array}{ccl}
\dd \left(
\begin{array}{l}
{x}_{1}\\
{x}_{2}
\end{array}
\right) &=& \left(
\begin{array}{c}
-k{x}_{1}+\frac{k{x}_{1}^{\dag}\left(C_1^\text{in}-{x}_{1}\right)}{C_1^\text{in}-{x}_{1}^{\dag}}
+\left(C_1^\text{in}-{x}_{1}\right)u\\
k{x}_{1}-\frac{k{x}_{1}^{\dag}{x}_{2}}{C_1^\text{in}-{x}_{1}^{\dag}}-{x}_{2}u
\end{array}
\right) \dd t +\left(
\begin{array}{c}
-\sigma{x}_{1}\\
\sigma{x}_{1}
\end{array}
\right)
\dd \omega\\
y&=&\left({x}_{1}-{x}_{1}^{\dag}\right)\left(C_1^\text{in}-{x}_{1}\right)
\end{array}
\right.
\end{equation*}
Apparently, at the desired state the diffusion term will not vanish.
Therefore, it is impossible to find a {radially unbounded} Lyapunov function so
that the stochastic CRTR system is stable at the desired point in
the sense of probability.

{According to the conservation law, it is easy to get that the state $(x_1,x_2)^\top$ only evolves in the manifold $\left\{\left(x_{1},C_{1}^\text{in}-x_{1}\right), \forall x_{1}\right\}$. Therefore, based on \textit{Remark} \mbox{\ref{rem explaination of how to find decompation transformation}} the decomposition transformation can be defined as}
\begin{equation*}{\label{DeconmposeTransformationofNonlinearProcess}}
\bm{\Phi}(\bm{x}) = \left(
\begin{array}{c}
\bar{x}_{1}\\
\bar{x}_{2}
\end{array}
\right)= \left(
\begin{array}{c}
{x}_{1}\\
{x}_{1}+{x}_{2}
\end{array}
\right)
\end{equation*}
through which the system is decomposed into
\begin{eqnarray*}
\Sigma_{subS}:&~&\left\{
\begin{array}{lll}\label{NonlinearSubSystem1}
\dd \bar{{x}}_{1}&=&\left[-k\bar{{x}}_{1}+
\frac{k{x}_{1}^{\dag}\left(C_1^\text{in}-\bar{{x}}_{1}\right)}{C_1^\text{in}-{x}_{1}^{\dag}}
+\left(C_1^\text{in}-\bar{{x}}_{1}\right)u\right]\dd t -\sigma\bar{{x}}_{1}\dd \omega\\
{y}_1&=&\left(\bar{{x}}_{1}-{x}_{1}^{\dag}\right)\left(C_1^\text{in}-\bar{{x}}_{1}\right)
\end{array}
\right.
\end{eqnarray*} and
\begin{eqnarray*}
\Sigma_{subD}:&~&\left\{
\begin{array}{lll}\label{NonlinearSubSystem2}
\dd\bar{{x}}_{2}&=& (C_1^\text{in}-\bar{{x}}_{2})q\dd t \\
{y}_{2}&=&{s}_{2}(\bar{{x}}_{2},{u})
\end{array}
\right.
\end{eqnarray*}
Note the facts that
$C_1^\text{in}={x}_{1}(0)+{x}_{2}(0)=\bar{x}_2(0)$ and the volume of
inflow or outflow is finite within a finite interval, i.e.,
$\int_{0}^{t}q(\tau)\dd \tau\leq \infty~(\forall~t\textgreater 0)$,
so we get
\begin{equation*}
\bar{{x}}_{2}(t)=\bar{{x}}_2(0)+\text{e}^{-\int_{0}^{t}q(s)\dd
s}\left(\bar{{x}}_{2}(0)-\bar{{x}}_{2}(0)\right)=\bar{{x}}_2(0)
\end{equation*}
which means the subsystem $\Sigma_{subS}$ following
\begin{equation*}
\dd \bar{{x}}_{2}=0\dd t
\end{equation*}
Hence, the vector field of $\Sigma_{subD}$ equals zero and
$\boldsymbol{\Phi}(\bm{x})$ is a decomposition transformation of the
CSTR process system.

Construct the Lyapunov function to be
$V(\bar{{x}}_{1})=\frac{1}{2}\left(\bar{{x}}_{1}-{x}_{1}^{\dag}\right)^{2}$,
then we have
\begin{eqnarray*}
\mathcal{L}[V(\bar{{x}}_{1})]&=&\left(\bar{{x}}_{1}-{x}_{1}^{\dag}\right)
\left[-k\bar{{x}}_{1}+
\frac{k{x}_{1}^{\dag}\left(C_1^\text{in}-\bar{{x}}_{1}\right)}{C_1^\text{in}-{x}_{1}^{\dag}}
+\left(C_1^\text{in}-\bar{{x}}_{1}\right)u\right]+\frac{1}{2}\sigma^{2}\bar{{x}}_{1}^{2}\\
&=&
yu-\frac{kC_1^\text{in}}{\left(C_1^\text{in}-{x}_{1}^{\dag}\right)}\left(\bar{{x}}_{1}-{x}_{1}^{\dag}\right)^{2}
+\frac{1}{2}\sigma^{2}\bar{{x}}_{1}^{2}
\end{eqnarray*}
Let
$\delta=\frac{kC_1^\text{in}}{2\left(C_1^\text{in}-{x}_{1}^{\dag}\right)}$,
then the above equation changes to be
\begin{eqnarray*}
\mathcal{L}[V(\bar{{x}}_{1})] &=&
yu-2\delta\left(\bar{{x}}_{1}-{x}_{1}^{\dag}\right)^{2}+\frac{1}{2}\sigma^{2}\bar{{x}}_{1}^{2}
\end{eqnarray*}
Hence, for any $\|\bar{\bm{x}}_{1}-\bm{x}_{1}^{\dag}\| \geq R$ where
$$R=\frac{\sigma^2+\sigma\sqrt{2\delta}}{2\delta-\sigma^2}x_1^\dag$$ we have
\begin{eqnarray*}
\mathcal{L}[V(\bar{\bm{x}}_{1})] &\leq& yu
-\delta\|\bar{\bm{x}}_{1}-\bm{x}_{1}^{\dag}\|^{2}
\end{eqnarray*}
This means that the subsystem $\Sigma_{subD}$ is strictly state
stochastic weak passive with respect to the {radially unbounded} storage function
$V(\bar{{x}}_{1})$. Note that the stochastic passive radius $R$ is
quite small due to $\sigma\ll k$.

Additionally, there exists a positive constant $\epsilon$ such that
$0\textless\epsilon<\text{min}\{{x}_1^{\dag}-{R},C_1^\text{in}-R-{x}_{1}^{\dag}\}$
and for any $\|\bar{{x}}_{1}-{x}_{1}^{\dag}\| \textless
{R}+\epsilon$ the noise port $\sigma^{2}\bar{{x}}_{1}^{2}$ is
nonsingular. Thus, based on \textit{Theorem} \ref{RefinedSWPT}, the
system can be locally stochastically asymptotically weakly stable under any
negative proportional controller.

\begin{figure}
  \centering
  \includegraphics[width=1\textwidth]{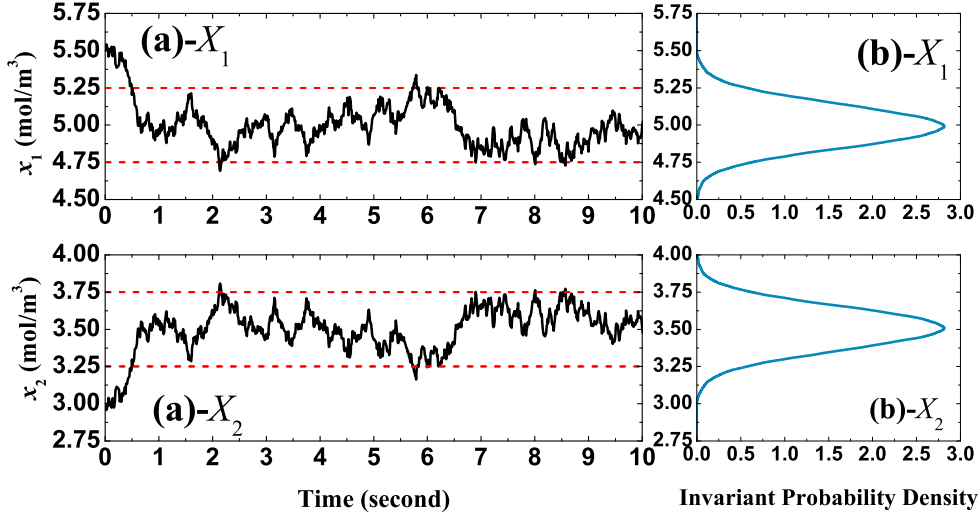}
  \caption{Time evolution (a) of the state of the CSTR process and the corresponding invariant probability density (b) without controller implemented. The bounded areas in (a) are $5.0\pm 0.25$ for $x_1$ and $3.5\pm 0.25$ for $x_2$.}
\end{figure}

{To better exhibit the stochastic asymptotic weak stability, some simulations are made on the above CSTR process with the initial state designated as $\bm{x}(0)=(5.5,3)^{\top}$, the desired state as $\bm{x}^{\dag}=(5,3.5)^{\top}$, and the disturbance as $\sigma=0.03$. The other parameters are $k=1$ mole/m$^3$/s, $q(0)=0.33$ m$^3$/s and $C_1^\text{in}=8.5$ mole/m$^3$. Fig. $1$ shows the time evolution of the state
$\bm{x}(t)=(x_1(t),x_2(t))^\top$ without control and the corresponding invariant probability density functions. The ranges, $5.0\pm 0.25$ for $x_1$ and $3.5\pm 0.25$ for $x_2$, bounded by two dotted lines, respectively, represent the areas in which the state evolves with probability $90\%$. When the controller $u=-y_{1}$ is implemented on this process, the state evolution and the invariant probability density functions will fasten on the region around the desired state more. Shown in Fig. $2$ are the results. We also use two dotted lines to bound the areas in which the state evolves with $90\%$. Now, they change to be $5.0\pm 0.1$ for $x_1$ and $3.5\pm 0.1$ for $x_2$, which is clearly more intensive around the desired state. This information can be also observed from Fig. $2$(b) in which the invariant probability density function changes ``thinner" around the desired state than the corresponding one appearing in Fig. $1$(b). The convergent behaviors of the invariant probability density functions under the control are exhibited in Fig. $3$. To observe more detailed convergent process, we only exhibit the simulation from $t=0$ s to $t=3$ s. As can be seen in Fig. $3$, the initial probability density functions, at $t=0$, deviate the desired state too much, but they will converge to the invariant probability density functions that fasten on the small region around the desired state as the controller is put into force.}

\begin{figure}
  \centering
  \includegraphics[width=1\textwidth]{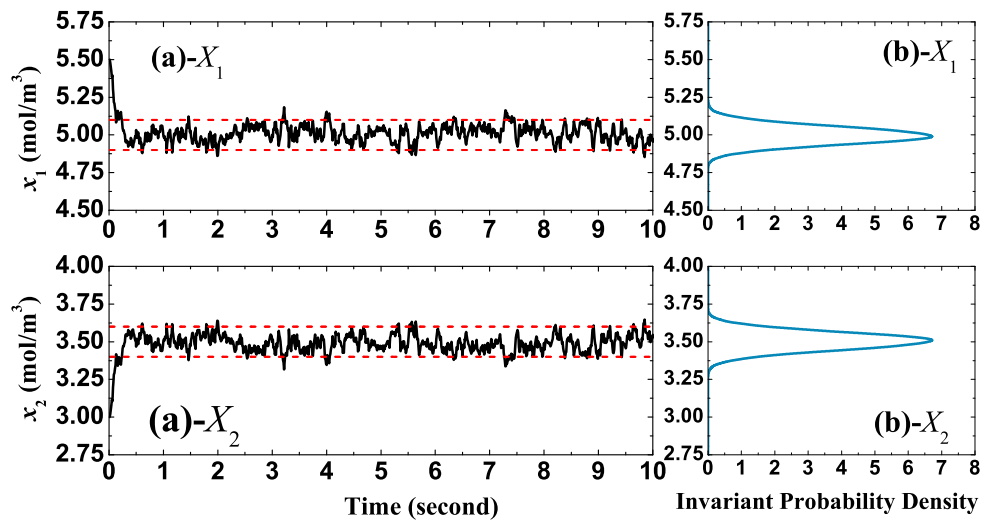}
    \caption{Time evolution (a) of the state of the CSTR process and the corresponding invariant probability density (b) with the controller $u=-y_{1}$ implemented. The bounded areas in (a) are $5.0\pm 0.1$ for $x_1$ and $3.5\pm 0.1$ for $x_2$.}
\end{figure}

\begin{figure}
  \centering
  \includegraphics[width=1\textwidth]{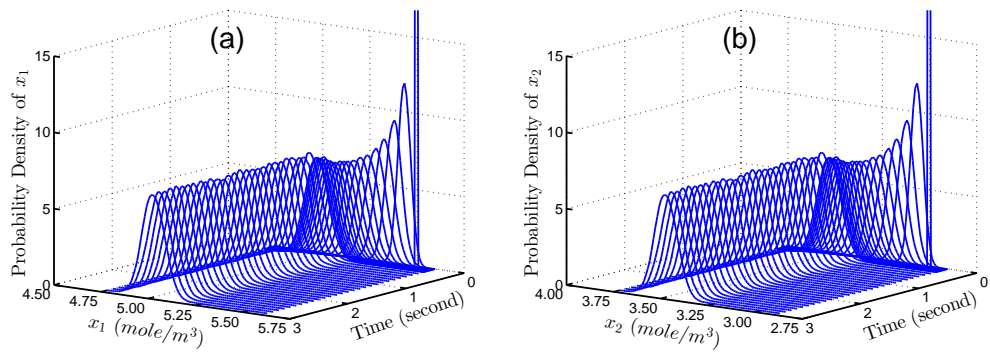}
  \caption{Convergent behaviors of probability density functions (a) for $X_1$ and (b) for $X_2$ with the controller $u=-y_{1}$ implemented.}
\end{figure}

\section{Conclusions and future research}
This work has presented a theoretical framework of stochastic weak
passivity serving for stabilizing the stochastic differential systems
with nonvanishing noise. The main contributions include: i) deriving the necessary
conditions to say a stochastic system stochastically passive or the
sufficient conditions that a stochastic system must lose stochastic
passivity; ii) proving that it is impossible for some stochastic
systems to be stabilized in probability; iii) defining a new concept
of stochastic weak passivity to serve for those systems losing
stochastic passivity, which captures the stochastic passivity of the
system not in the whole state space but only outside a ball centered
around the desired state; iv) associating stochastic weak passivity
to asymptotic weak stability of systems, and further providing the sufficient
conditions for global and local asymptotic weak stabilization of
nonlinear stochastic differential systems by means of negative
feedback laws.

The stochastic weak passivity provides an alternative way to stabilize the
transition measure {as well as capturing the ergodicity} of the stochastic differential systems with
nonvanishing noise. However, there is still a large room for this
method to be improved or expanded. {An important issue is that the whole theoretical framework works under the assumption that the stochastic term $\boldsymbol{\omega}$ is a standard Wiener process. The motivation of making such an assumption is that the current concept is developed based on the stochastic passivity. For the latter, the noise term is assumed as a standard Wiener process\mbox{\cite{Florchinger99}}. However, the standard Wiener process is just a kind of ideal noise, and is used mainly for the simplicity of analysis. As far as many practical systems are concerned, this ideal noise is not accurate enough to represent the internal modeling uncertainty. Therefore, it is interesting but challenging to use other stochastic processes instead of the standard Wiener process for developing stochastic weak passivity. Towards this task, the infinitesimal generator $\mathcal{L[\cdot]}$ of Eq. \mbox{(\ref{InfinitesimalGenerator})} needs to be redefined accordingly.} In addition, other possible points of future research include: i) weakening the condition of nonsigularity of the
diffusion matrix $\bm{h}_{1}(\bar{\bm{x}}_{1},\bm{u})\bm{h}_{1}^{\top}(\bar{\bm{x}}_{1},\bm{u})$;
ii) applying the stochastic weak passivity theory to some special
stochastic differential systems, such as stochastic affine systems,
thermodynamic process systems, and drive the development of these
fields in control techniques; iii) developing the determinist
version of stochastic weak passivity.


\pagestyle{myheadings}
\thispagestyle{plain}
\markboth{ ZHOU FANG, AND CHUANHOU GAO}{stochastic weak passivity}

\end{document}